\documentclass[12pt,psamsfonts,reqno]{amsart}

\usepackage{amssymb,amsfonts}
\usepackage[all,arc]{xy}
\usepackage{enumitem}
\setlist[enumerate]{topsep=2pt, itemsep=2pt}

\usepackage{hyperref}
\usepackage[top=1in, bottom=1.25in, left=1.25in, right=1.25in]{geometry}
\usepackage{amsfonts}
\usepackage{amsthm}
\usepackage{amsmath}
\allowdisplaybreaks

\usepackage{parskip}
\usepackage{graphicx}
\usepackage{amssymb}
\usepackage{xtab}
\usepackage{mathtools}

\usepackage{cleveref}
\usepackage{mathrsfs}
\usepackage{tikz} % https://q.uiver.app/ 
\usepackage{tikz-cd}
\usepackage{float}
\usetikzlibrary{positioning}
\usetikzlibrary{shapes.geometric,arrows}
\usepackage{mnotes}

\usepackage[none]{hyphenat}
\usepackage{float}
\newtheorem{thm}{Theorem}[section]
\newtheorem{thrm}[thm]{Theorem}
\newtheorem{cor}[thm]{Corollary}
\newtheorem{prop}[thm]{Proposition}
\newtheorem{lem}[thm]{Lemma}

\theoremstyle{definition}

\theoremstyle{remark}
\newtheorem{rem}[thm]{Remark}


\numberwithin{equation}{section}

\let\originalleft\left
\let\originalright\right
\renewcommand{\left}{\mathopen{}\mathclose\bgroup\originalleft}
\renewcommand{\right}{\aftergroup\egroup\originalright}

\DeclareMathOperator{\pic}{Pic}

\DeclareMathOperator{\gal}{Gal}

\DeclareMathOperator{\spec}{Spec}

\DeclareMathOperator{\Bl}{Bl}
\DeclareMathOperator{\Br}{Br}
\DeclareMathOperator{\Bra}{Br_a}
\DeclareMathOperator{\inv}{inv}

\DeclareMathOperator{\ind}{Ind}
\DeclareMathOperator{\im}{im}
\DeclareMathOperator{\Ho}{H}

\DeclareMathOperator{\ab}{ab}

\DeclareMathOperator{\Gal}{Gal}

\DeclareMathOperator{\Hom}{Hom}

\DeclareMathOperator{\Pic}{Pic}

\DeclareMathOperator{\Res}{Res}

\newcommand{\gm}{\mathbb{G}_m}

\def\Z{\ifmmode{{\mathbb Z}}\else{${\mathbb Z}$}\fi}
\def\Q{\ifmmode{{\mathbb Q}}\else{${\mathbb Q}$}\fi}
\def\C{\ifmmode{{\mathbb C}}\else{${\mathbb C}$}\fi}
\def\P{\ifmmode{{\mathbb P}}\else{${\mathbb P}$}\fi}
\def\H{\ifmmode{{\mathrm H}}\else{${\mathrm H}$}\fi}
\def\G{\ifmmode{{\mathbb G}}\else{${\mathbb G}$}\fi}
\def\R{\ifmmode{{\mathbb R}}\else{${\mathbb R}$}\fi}
\def\F{\ifmmode{{\mathbb F}}\else{${\mathbb F}$}\fi}
\def\N{\ifmmode{{\mathbb N}}\else{${\mathbb N}$}\fi}
\def\O{\ifmmode{{\calO}}\else{${\calO}$}\fi}
\def\D{\ifmmode{{\cal{D}}^b}\else{${{\cal{D}}^b}$}\fi}

\DeclareMathOperator{\coker}{coker}
\DeclareMathOperator{\rk}{rk}

\DeclareMathOperator{\Ind}{Ind}
\DeclareMathOperator{\Cor}{cor}

\DeclareMathOperator{\class}{class}

\DeclareMathOperator{\et}{\acute{e}t}
\DeclareMathOperator{\sing}{sing}

\DeclareMathOperator{\res}{res}

\DeclareMathOperator{\br}{Br}

\DeclareMathOperator{\h}{H}
\DeclareMathOperator{\cxan}{cx}

\newcommand{\z}{\mathbb{Z}}
\newcommand{\q}{\mathbb{Q}}

\newcommand{\calO}{{\mathcal O}}

\providecommand{\ab}{\mathrm{ab}}

\providecommand{\Pic}{\mathrm{Pic}}

\def\H{\ifmmode{{\mathrm H}}\else{${\mathrm H}$}\fi}

\newcommand*{\myproofname}{Proof of \Cref{thm:ShaTrivial}}

\theoremstyle{definition}

\DeclareFontFamily{U}{wncy}{}
\DeclareFontShape{U}{wncy}{m}{n}{<->wncyr10}{}
\DeclareSymbolFont{mcy}{U}{wncy}{m}{n}
\DeclareMathSymbol{\Sh}{\mathord}{mcy}{"58} 
\DeclareMathSymbol{\Be}{\mathord}{mcy}{"42} 

\title[Brauer groups of certain affine cubic surfaces]{Brauer groups of certain affine cubic surfaces}
\author[A. Alfaraj]{Abdulmuhsin Alfaraj}
\address{Abdulmuhsin Alfaraj\\
	Center for Intelligent Secure Systems \\
	King Fahd University of Petroleum \& Minerals \\
	Dhahran \\
	31261 \\
	KSA.}
\email{abdulmuhsin.faraj@kfupm.edu.sa}

\begin{document}
	
	\begin{abstract}
		We study the Brauer groups of affine surfaces that are complements of singular hyperplane sections of smooth cubic surfaces over a field $k$ of characteristic $0$. We determine the Brauer group over the algebraic closure as a Galois module for all the possible singular hyperplane sections. For the case when the hyperplane section is geometrically the union of three lines, we give explicit examples where transcendental elements of order $2$ and $3$ exist over $\q$. We end with an explicit example of an integral Brauer--Manin obstruction to the Hasse principle given by an element of order~$3$ on an affine cubic surface with singular boundary.
	\end{abstract}
	\maketitle
	\tableofcontents
	\setcounter{equation}{0}
	\setcounter{thm}{0}

	\section{Introduction}
	The Brauer group of a variety is a geometric object rich with arithmetic information. It is a birational invariant for smooth proper varieties, hence its importance in geometry. In 1970, Manin \cite{Manin71} successfully used the Brauer group to explain failures of the Hasse principle at the time, which sparked arithmetic interest in the method giving birth to the study of Brauer--Manin obstructions on varieties. Naturally, to study applications of the Brauer group one would first attempt to compute it. Determining the Brauer group, and especially its transcendental part, is a difficult problem in general. If the variety is not proper, the problem becomes even more challenging: indeed, one expects the Brauer group to be larger than that of any smooth compactification.
	
	A special class of affine varieties over number fields that attracted arithmetic interest is that of log K3 surfaces. Conjecturally, it is expected that the positivity of the log anticanonical class of an affine variety controls the abundance of integral points: the more positive it is, the more integral points one expects. Log K3 surfaces, having a trivial anticanonical class, are the borderline case, hence their significance. Moreover, as these surfaces are simply connected and can be embedded in smooth proper varieties where the boundary is a simple normal crossing divisor (see \cite[Definition 2.1]{Harpaz}), computing their Brauer groups is expected to be generally more approachable.

	The arithmetic of affine cubic surfaces, an important
	source of log K3 examples, has received considerable
	attention. Browning and Wilsch \cite{BW} develop
	heuristics and numerical investigations for the density of integral
	points on affine cubic surfaces. Colliot-Th\'el\`ene and Wittenberg \cite{CTOW}, and Lyczak,
	Mitankin, and Uppal \cite{LMU} study Brauer groups and integral
	Brauer--Manin obstructions for families of diagonal affine cubic
	surfaces, and Uppal \cite{Uppal} studied symmetric affine cubic surfaces, where in both families the boundaries are smooth genus $1$ curves. More generally,
	Bright and Lyczak \cite{BrightLyczak} give uniform bounds for the Brauer
	group modulo constants of complements of smooth anticanonical divisors
	in del Pezzo surfaces of degree at most $7$. In these settings, the
	Brauer group is closely related to the arithmetic of the smooth genus
	$1$ boundary.  A notable example in the singular case is that of Markoff surfaces, where the hyperplane section is the union of three rational lines, was studied by Ghosh and Sarnak \cite{GS}; Colliot-Thélène,  Wei, and Xu  \cite{CTWX}; Loughran and Mitankin \cite{MarkoffSurf}; and Dao \cite{Dao}.
	
	These works motivate the systematic study of Brauer groups of
	complements of singular hyperplane sections in smooth cubic surfaces, which we develop in this paper. In addition to their arithmetic importance, these surfaces  are also interesting from a geometric perspective, since they can admit infinite discrete automorphism groups. Such infinite automorphism groups have been studied and used for arithmetic applications recently,  e.g.\ \cite{KollarLi2024} and \cite{KollarVillalobos2024}. In the particular case of Markoff surfaces, these automorphism groups were combined with a Brauer--Manin obstruction in \cite[\S 5.2]{CTWX}.

	For this class of affine cubic surfaces, we compute the Brauer group over the algebraic closure as a Galois module. In this situation, the Brauer group is related to the combinatorics of the irreducible (genus $0$) components of the boundary divisor and their intersection.  More precisely, we prove the following. 
	
	\begin{thrm}\label{thm:Main}
		Let $k$ be a field of characteristic $0$ and let $\Gamma_k:=\gal(\overline{k}/k)$.
		Let $X$ be a smooth cubic surface over $k$, $H$ be a singular hyperplane section, and $U:=X\setminus H$. For $d\in k$, let $M_d=(\Ind_{k(\sqrt{d})/k} \z) / \z$, where $\ind_{k(\sqrt{d})/k}$ is the induced module from $\Gamma_{k(\sqrt{d})}$ to $\Gamma_k$.
		\begin{enumerate}[label=(\roman*),itemsep=2pt,parsep=2pt,topsep=4pt]
			\item Suppose that $H$ is the union of a line $\ell$ and a smooth conic $C$ over $k$. Then 
			\[  \br \overline{U} \cong  
			\begin{cases}
				0, & \text{if $\ell$ is tangent to C,} \\
				\q/\z(-1), & \text{if $|\ell(k)\cap C(k)|=2,$} \\ 	\varinjlim \left( {M_d}/{nM_d}(-1) \right), & \text{if $\ell\cap C=\spec(k(\sqrt{d}))$, $d\notin k^2$}.
			\end{cases} 
			\]
			
			\item Suppose that $H$ is a geometrically irreducible singular cubic curve $C$. Then 
			\[  \br \overline{U} \cong  
			\begin{cases}
				0, & \text{if $C$ is cuspidal,} \\
				\q/\z(-1), & \text{if $C$ is nodal split multiplicative,}\\ 	
				\varinjlim \left( {M_d}/{nM_d}(-1) \right), & \text{if $C$ is nodal splitting over $ k(\sqrt{d})$, $d\notin k^2$}.
			\end{cases} 
			\]
			
			\item Suppose that $\overline{H}$ is the union of three lines $\ell_1,\ell_2,\ell_3$ with splitting field $L$. Then
			\[ \br \overline{U} \cong 
			\begin{cases}
				0, & \text{if the lines meet at an Eckardt point,} \\
				\q/\z(-1), & \text{if $\cap \ell_i =\emptyset$ and  $[L:k]\in \{1,3\}$,} \\
				\varinjlim \left( {M_d}/{nM_d}(-1) \right), & \text{if $\cap \ell_i =\emptyset$ and  $k(\sqrt{d})\subset L$, for some $d \in k\setminus k^2$  }. \\
			\end{cases} 
			 \]	
		\end{enumerate}
	\end{thrm} 
	\vspace{4pt}
	
	For number fields, these computations also imply that the Brauer group of $U$
	modulo constants is finite, with an explicit bound depending only on the
	degree of the field $k$ (see Corollary \ref{cor:transcendental-bound}).
	
	For case $(iii)$ of Theorem \ref{thm:Main}, we construct examples where the transcendental part of $\Br(U)$ is non-trivial over $k=\q$.
	\begin{thrm}\label{thm:DescentToQ}
		For each $n\in\{2,3\}$, there exists a smooth affine cubic
		surface $U$ over $\mathbb{Q}$ that is the complement of
		three geometric lines such that every algebraic Brauer class is constant and
		$\Br(U)$ contains a transcendental element of order $n$.
	\end{thrm}
	For the case $n=3$, we also show that there exists an integral model of $U$ on which there is an integral Brauer--Manin obstruction to strong approximation explained by the order $3$ transcendental element (see Proposition \ref{prop:BrTrSa}).

	We end the paper with an example of an integral Brauer--Manin obstruction to the  Hasse principle on a non-diagonal affine cubic surface where the boundary is geometrically the union of three lines meeting at an Eckardt point. To our knowledge, this is the first explicit example of an integral Brauer--Manin obstruction to the  Hasse principle explained by an element of order~$3$ on the complement of a singular hyperplane section of a smooth cubic surface. For comparison, the integral Brauer--Manin obstructions exhibited for Markoff surfaces come from Brauer classes of even order \cite{CTWX, MarkoffSurf}, while the affine diagonal cubic surfaces of \cite{CTOW, LMU} are complements of smooth hyperplane sections.
	
	Consider the affine cubic surface given by
	\begin{equation*}
		\mathcal{U}:\> 9 x^3+ y^3 = z^2 + 3 \quad \subset \mathbb{A}^3_{\z}.
	\end{equation*}
	Grechuk \cite{gr24} posted a MathOverflow question asking whether an integral solution exists for the equation defining $\mathcal{U}$. Shatrov \cite{gr24}, using elementary arguments, has shown that there are indeed no integral solutions for the equation. We show that there is an integral Brauer--Manin obstruction to the existence of integral points on the affine cubic surface $\mathcal{U}$ coming from an order $3$ element in $\Br(\mathcal{U}_{\q})$, where $\mathcal{U}_{\q}=\mathcal{U}\times \q$.
	\begin{thm}\label{thm:IMBO}
		Let $\omega$ be a primitive cube root of unity. Then $\Br(\mathcal{U}_{\q})/\Br(\q) \cong \z/3$ and is generated by the class of the degree $3$ cyclic algebra
		$$B:= \Cor_{\q(\mathcal{U}_{\q})(\omega)/\q(\mathcal{U}_{\q})} \left( -9, z+\sqrt{-3} \right)_{\omega} \text{ in } \Br(\q(\mathcal{U}_{\q}))[3].$$
		Moreover, the set of points in $\mathcal{U}(\mathbb{A}_{\z})$ orthogonal to $B$ is empty, i.e., $$\mathcal{U}(\mathbb{A}_{\z})^B = \emptyset.$$
	\end{thm}

	\vspace{10pt}
	
	\noindent\textbf{Notation.} Let $k$ be a field of characteristic $0$ and let $\overline{k}$ be an algebraic closure of $k$. Let $\Gamma_k:=\Gal(\overline{k}/k)$ be the absolute Galois group of $k$. For a $k$-variety $X$, we write $\overline{X}:=X\times_k \overline{k}$ and $X_L:=X\times_k L$ for any field extension $L/k$. 
	We let $\Br(X)=\Ho_{\et}^2(X,\gm)$ denote the Brauer group of a scheme $X$.
	If $X$ is a $k$-variety, we write $\Br_1(X)=\ker \left[\Br(X)\rightarrow \Br(\overline{X})\right]$ for the algebraic Brauer group, we write ${\Br_0(X)=\im \left[\Br(k)\rightarrow \Br(X)\right]}$ for the constant part, we call $\Br(X)/\Br_1(X)$ the transcendental Brauer group, and we let ${\Bra(X):=\Br_1(X)/\Br_0(X)}$. Finally, we write $\z/n$ for the abelian group $\z/n\z$.
	
	\vspace{10pt}
	\noindent\textbf{Acknowledgements.} I would like to thank my supervisor Daniel Loughran for his constant support and guidance. I would also like to thank Olivier Wittenberg and Julian Demeio for useful discussions. Finally, I would like to thank the referee for their valuable comments that greatly improved the paper.

	\section{The algebraic Brauer group}
	
	In this section, we study the algebraic part of the Brauer group of the complement of a singular hyperplane section in a smooth cubic surface. We show that the algebraic part is absolutely bounded independently  of the ground field (see Lemma \ref{lem:BrUisoH1}). We also compute all the possibilities for the algebraic Brauer group when the hyperplane section is geometrically the union of three lines (see Proposition \ref{prop:AlgBr3lines}).
	
	We start by listing all the possible singular hyperplane sections.
	\begin{lem}\label{lem:HXposs}
		Let $X\subset \mathbb{P}^3$ be a smooth cubic surface over a field $k$ of characteristic $0$ and let $H$ be a hyperplane section. If $H$ is singular, then it is one of the following:
		\begin{enumerate}[label=(\roman*),itemsep=2pt,parsep=2pt,topsep=4pt]
			\item a smooth conic and a line over $k$;
			\item a geometrically integral singular cubic curve over $k$;
			\item three geometric lines intersecting at a single (Eckardt) point;
			\item three geometric lines with distinct pairwise intersection points   where the Galois group of their splitting field over $k$
			is isomorphic to $0$, $\z/2$, $\z/3$, or $S_3$.
		\end{enumerate}
	\end{lem}

	We now prove that the Picard group of the complement of a singular hyperplane section is torsion-free and we use that to give a description for the algebraic Brauer group. We start by stating the following relevant lemma.
	\begin{lem}[{{\cite[Lemma 2.1]{MarkoffSurf}}}]\label{lem:UinvertibleGlobal}
		Let $X$ be a smooth cubic surface over $k$, $H$ be a hyperplane section and  $U:=X\setminus H$. Then ${k}[U]^\times=k^\times$.
	\end{lem}

	\begin{lem}\label{lem:BrUisoH1}
		Let $X$ be a smooth cubic surface over a field $k$ of characteristic $0$, $H$ be a hyperplane section and $U:=X\setminus H$. Then the following statements hold.
		\begin{enumerate}[label=(\roman*),itemsep=2pt,parsep=2pt,topsep=4pt]
			\item $\pic(\overline{U})$ is torsion-free.
			\item If $k$ is a number field or if $U(k)\neq \emptyset$, then
			$$\Bra(U)\cong \Ho^1(k, \pic(\overline{U})).$$
			Moreover, we have
			$$
			|\Bra(U)|\leq |W(E_6)|^6,
			$$
			where $W(E_6)$ is the Weyl group of the root system $E_6$.
		\end{enumerate}
	\end{lem}
	
	\begin{proof}
		(i) We show that $\pic(\overline{U})$ is torsion-free. Write 
		$$ \overline{H}=H_1\cup\cdots\cup H_r $$
		for the decomposition of $\overline{H}$ into its irreducible components. Clearly, $r\leq 3$ by Lemma \ref{lem:HXposs}.
		Then as ${k}[U]^\times=k^\times$ by Lemma \ref{lem:UinvertibleGlobal}, we have an exact sequence of $\Gamma_k$-modules
		\begin{equation}\label{eq:PicUexact}
			 0\rightarrow \bigoplus_{i=1}^r \z [H_i] \rightarrow \pic(\overline{X}) \rightarrow \pic (\overline{U})\rightarrow 0.
		\end{equation}
		Therefore, it suffices to show that if $[D]\in \pic(\overline{X})$ satisfies $n[D] \in \bigoplus_{i=1}^r \z [H_i]$ for some $n>0$, then $[D] \in \bigoplus_{i=1}^r \z [H_i]$. We show that this holds for all the cases in Lemma \ref{lem:HXposs}. For cases $(iii)$ and $(iv)$ in Lemma \ref{lem:HXposs}, this follows by \cite[Proposition 2.2]{CTWX}. 
		
		Suppose now that $H$ is geometrically integral so that $\overline{H}$ is irreducible. Realizing $\overline{X}$ as the blow up of $\mathbb{P}^2_k$ at $6$ points in general position, we may write 
		$$\Pic(\overline X)
		\cong \mathbb Z\ell \oplus \bigoplus_{i=1}^{6}\mathbb Ze_i $$ 
		where $\ell$ is the pullback of the class of a line
		and $e_1,\ldots,e_6$ are the classes of the exceptional curves; moreover, we have that $$[H]=-[K_X]=3 \ell - \sum_{i=1}^6 e_i$$
		where $K_X$ is the canonical class (see e.g.\ \cite[Chapter V, Proposition 4.8]{Har}). Now suppose that $[D]\in \pic(\overline{X})$ satisfies $n[D]=m [H]$ for some $n,  m \in \z_{>0}$. Then it is easy to see that $n\mid m$ so that $[D]\in \z [H]$. 
		
		Suppose now that $\overline{H}$ is the union of a smooth conic $C$ and a line $L$ over $k$. Since $[H]=[C]+[L]$, we have 
		$$\z [C] \oplus \z [L] = \z [K_X] \oplus \z[L].$$
		Since the Weyl group $W(E_6)$ acts transitively on the $27$ lines of $\overline{X}$ (see e.g.\ \cite[Corollary 26.7]{ManinCubic}) while fixing $[K_X]$, we may assume that $[L]=e_1$. Now suppose that $[D]\in \pic(\overline{X})$ satisfies $n[D]= m e_1+ t (3 \ell - \sum_{i=1}^6 e_i)$ for some $n>0$ and $ m, t \in \z$. Then $n[D]= (m -t) e_1+ t (3 \ell - \sum_{i=2}^6 e_i)$ implies that $n$ divides both $m-t$ and $t$. Thus, $n$ divides both $m$ and $t$, which gives that $[D]\in \z [C] \oplus \z[L]$, as desired.

		(ii) If  $U(k)\neq \emptyset$ or if $k$ is a number field, then by \cite[Lemma 6.3(iii) and (iv)]{Sansuc} and Lemma \ref{lem:UinvertibleGlobal} we have 
		$$\Br_1(U)/\Br_0(U)\cong \Ho^1(k, \pic(\overline{U})).$$
		If $K$ is the minimal extension over which $\pic(\overline{X})$ splits, then 
		$\Ho^1(K, \pic(\overline{U}))=0$ since $\pic(\overline{U})$ is a torsion-free  $\Gamma_K$-module. By the inflation-restriction sequence, we obtain
		$$\Ho^1(k, \pic(\overline{U}))\cong \Ho^1(K/k, \pic({U}_K)).$$
		
		Since $K$ was chosen to be minimal, $\Gal(K/k)$ acts on the lattice $\pic(\overline{U})$ via some subgroup $G$ of the Weyl group $W(E_6)$. Let $M:=\pic(\overline{U})$ and $n:=|G|$. As $G$ is a finite group and $M$ is a $G$-module, we have that $H^1(G,M)$ is finite and annihilated by $n$ (see e.g.\ \cite[proof of Corollary 1.3.6]{CTSK}). Since $M$ is torsion-free and $n$ annihilates $ H^1(G,M)$, taking the cohomology of the exact
		sequence
		\[
		0\longrightarrow M\xrightarrow{\,n\,}M
		\longrightarrow M/nM\longrightarrow 0
		\]
		gives a surjection
		\[
		(M/nM)^G\twoheadrightarrow H^1(G,M).
		\]
		Consequently,
		\[
		| H^1(G,M) |
		\leq |(M/nM)^G|
		\leq |M/nM|
		=|G|^{\rk M}\leq |W(E_6)|^6,
		\]
		which completes the proof.
	\end{proof}
	
	\begin{rem}
		In practice, the order of $\Bra(U)$ is expected to be much less than the bound obtained in Lemma \ref{lem:BrUisoH1}. For example, if $H$ is geometrically irreducible, then the order of $\Bra(U)$ is at most $256$ (see \cite[Proposition 2.1 and its proof]{BrightLyczak}).
	\end{rem}

	We now compute the possibilities for the algebraic Brauer group when the hyperplane section is geometrically the union of three lines.
	\begin{prop}\label{prop:AlgBr3lines}
		Let $X\subset \mathbb{P}^3_k$ be a smooth projective cubic surface over a field $k$ of characteristic $0$. Suppose a plane $\mathbb{P}^2_k\subset \mathbb{P}^3_k$ cuts on $\overline{X}$ three lines $\ell_1,\ell_2,\ell_3$ over $\overline{k}$, and let $U\subset X$ be the complement of this plane. Assume that either $U(k)\neq\emptyset$ or that $k$ is a number field.
		
		\begingroup
		\renewcommand{\arraystretch}{1.2}
		
		\begin{enumerate}[itemsep=2pt,parsep=2pt,topsep=4pt]
			\item If $\{\ell_1,\ell_2,\ell_3\}$ is a single $\Gamma_k$-orbit, then we have the following possibilities for $\Bra(U)$ with corresponding $\Bra(X)$:
			
			\begin{center}
					\medskip
				\begin{tabular}{|c|c|c|c|c|c|c|c|c|c|}
					\hline
					$\Bra(U)$
					& $0$ & $\z/2$ & $\z/2$
					& $(\z/2)^2$ & $(\z/2)^2$ & $(\z/2)^2$
					& $\z/4$ & $\z/3$ & $(\z/3)^2$ \\
					\hline
					$\Bra(X)$
					& $0$ & $0$ & $\z/2$
					& $0$ & $\z/2$ & $(\z/2)^2$
					& $\z/2$ & $\z/3$ & $(\z/3)^2$ \\
					\hline
				\end{tabular}
				\medskip
			\end{center}
			
			\item If $\{\ell_1,\ell_2,\ell_3\}$ consists of two $\Gamma_k$-orbits then we have the following possibilities for $\Bra(U)$ with corresponding $\Bra(X)$:
			
			\begin{center}
				\medskip
				\begin{tabular}{|c|c|c|c|c|c|c|}
					\hline
					$\Bra(U)$
					& $0$ & $\z/2$ & $\z/2$
					& $(\z/2)^2$ & $(\z/2)^2$ & $(\z/2)^2$ \\
					\hline
					$\Bra(X)$
					& $0$ & $0$ & $\z/2$
					& $0$ & $\z/2$ & $(\z/2)^2$ \\
					\hline
				\end{tabular}
				
				\medskip
				
				\begin{tabular}{|c|c|c|c|c|c|c|}
					\hline
					$\Bra(U)$
					& $(\z/2)^3$ & $(\z/2)^3$ & $(\z/2)^3$
					& $\z/4$ & $\z/2\times\z/4$ & $\z/2\times\z/4$ \\
					\hline
					$\Bra(X)$
					& $0$ & $\z/2$ & $(\z/2)^2$
					& $\z/2$ & $\z/2$ & $(\z/2)^2$ \\
					\hline
				\end{tabular}
				\medskip
			\end{center}
			
			\item If $\{\ell_1,\ell_2,\ell_3\}$ consists of three $\Gamma_k$-orbits, i.e.\ each $\ell_i$ is defined over $k$, then we have the following possibilities for $\Bra(U)$ with corresponding $\Bra(X)$:
		\end{enumerate}
		\vspace{-8pt}
		\begin{center}
			\begin{tabular}{|c|c|c|c|c|c|c|c|c|c|c|}
				\hline
				$\Bra(U)$
				& $0$ & $\z/2$ & $\z/2$
				& $(\z/2)^2$ & $(\z/2)^2$
				& $(\z/2)^3$ & $(\z/2)^3$ & $(\z/2)^4$
				& $\z/4$ & $\z/2\times\z/4$ \\
				\hline
				$\Bra(X)$
				& $0$ & $0$ & $\z/2$
				& $0$ & $\z/2$
				& $0$ & $\z/2$ & $(\z/2)^2$
				& $\z/2$ & $\z/2$ \\
				\hline
			\end{tabular}
			\medskip
		\end{center}
		\endgroup
	\end{prop}
	\begin{proof}
		First, we note that (3) was computed in \cite[Proposition 2.5]{MarkoffSurf}. We follow the same approach to compute the other cases.
		By Lemma \ref{lem:BrUisoH1}, we have that ${\br_1(U)/\br_0(U)\cong \h^1(k,\pic(\overline{U}))}$. We have an exact sequence 
		$$0\rightarrow \text{PDiv} \overline{X} \rightarrow \text{Lines} \overline{X} \rightarrow \pic \overline{X} \rightarrow 0$$ where $\text{Lines} \overline{X}$ is the free abelian group generated by the lines on $\overline{X}$ and $\text{PDiv} \overline{X}$ is the subgroup of principal divisors supported on the lines.
		This along with the exact sequence of $\Gamma_k$-modules \eqref{eq:PicUexact} gives a description of $\pic \overline{U}$ as the quotient of a $\Gamma_k$-permutation module by a $\Gamma_k$-submodule. We proceed as in \cite[Proposition 2.5]{MarkoffSurf} to compute $\h^1(k,\pic(\overline{U}))$ using Magma.
	\end{proof}	
	
	\begin{rem}
		The quotient
		$\Bra(U)/\Bra(X)$ in all cases of Proposition \ref{prop:AlgBr3lines} is killed by $2$. Indeed, this is straightforward except in (2) and (3) when {$\Bra(U)\cong\z/2\times\z/4$} and
		$\Bra(X)\cong\z/2$, in which cases our computations show that the Galois action on $\pic(\overline{X})$ factors through $D_8$, the dihedral group of order $8$. 
		Let $K$ be the minimal extension over which $\pic(\overline{X})$ splits and let $G=\gal(K/k)$. Taking the Galois cohomology of the exact
		sequence \eqref{eq:PicUexact} in Lemma \ref{lem:BrUisoH1} and applying Shapiro's lemma to the permutation module
		$\bigoplus_i\z
		[\ell_i]$ gives an exact sequence
		\[
		0\to\Bra(X)\to\Bra(U)\to
		\Ho^2\big(G,\textstyle\bigoplus_i\z[\ell_i]\big)
		\cong\bigoplus_{O}\Hom(G_O,\q/\z),
		\]
		where $O$ runs over the $\Gamma_k$-orbits of $\{\ell_1,\ell_2,\ell_3\}$ and $G_O\subseteq G$ is
		the stabilizer of a line in $O$. Let $G_O^{\ab}$ be the abelianization of $G_O$. Now, $\Hom(G_O,\q/\z)=\Hom(G_O^{\ab},\q/\z)$ is killed by the exponent of $G_O^{\ab}$. In cases (2) and (3) when $\Bra(U)\cong\z/2\times\z/4$ and
		$\Bra(X)\cong\z/2$, our computations show that each stabilizer $G_O$
		is either isomorphic to $D_8$ or $(\z/2)^2$. Since
		$D_8^{\ab}\cong(\z/2)^2$, we conclude that $\Bra(U)/\Bra(X)$ is killed by $2$. 
	\end{rem}

	\section{The Brauer group over $\overline{k}$}\label{sec:TransBr}
	In this section, we prove Theorem \ref{thm:Main} by determining $\br(\overline U)$ as a $\Gamma_k$-module for the different cases in Lemma \ref{lem:HXposs}. For each positive integer $n$, the Kummer sequence reduces the computation of $\br(\overline U)[n]$ to that of the cokernel of the map
	$$
	\H^2(\overline X,\mu_n)\longrightarrow \H^2(\overline U,\mu_n)
	$$
	(see Lemma \ref{lem:BrCoker}).
	We compute these cokernels using Gysin sequences for the different possible singular hyperplane sections.
	
	In the following preliminaries section, we establish the cohomological tools needed for these computations. We prove vanishing results that simplify the Gysin sequences and describe the effect of removing finitely many points from the surface (Lemmas \ref{lem:H4Utrivial} and \ref{lem:SminPoints}). We also relate the Gysin residue maps associated with two points on the surface and the one associated with the two points and a line in the surface (Lemma \ref{lem:CommDiag}). This last result is needed in the computations of the geometric three line case. 
	
	\subsection{Preliminaries}
	Let $X$ be a variety over a field $k$. Let $Z$ be a closed subscheme of $X$, and let $U:=X\setminus Z$. For a sheaf $\mathcal{F}$ on $X_{\et}$, let
	$$\Gamma_Z(X,\mathcal{F}):=\ker(\Gamma(X,\mathcal{F})\rightarrow \Gamma(U,\mathcal{F})),$$
	the group of sections of $\mathcal{F}$ supported on $Z$. Let $\H^r_{Z}(X,\mathcal{F})$, the cohomology of $\mathcal{F}$ with support on $Z$, 
	be the $r$th derived functor of the functor $\mathcal{G}\mapsto \Gamma_Z(X,\mathcal{G})$.
	
	\begin{thrm}[{{\cite[Theorem 9.4]{MilneEC}}}]\label{thm:exactAxiom}
		For a sheaf $\mathcal{F}$ on $X_{\et}$ and a closed subscheme $Z\subset X$, there is a long exact sequence
		$$\cdots \rightarrow \Ho^r_Z(X,\mathcal{F}) \rightarrow \Ho^r(X,\mathcal{F}) \rightarrow \Ho^r(U,\mathcal{F}) \rightarrow \Ho^{r+1}_Z(X,\mathcal{F}) \rightarrow \cdots.$$
		This sequence is functorial in the pair $X$ and $X\setminus Z$, and in $\mathcal{F}$.
	\end{thrm}
	
	The following is a consequence of Gabber's absolute purity theorem (see e.g.\ \cite[Theorem 2.3.1]{CTSK}).
	\begin{thrm}[{{\cite[Theorem 16.1]{MilneEC}}}]\label{thm:GabberPurity}
		Let $X$ be smooth. For any smooth closed $k$-subvariety  $Z\subset X$ of codimension $c$ and any locally constant sheaf $\mathcal{F}$ of $\z/n$-modules on $X$, there are canonical isomorphisms
		$$ \H^{r-2c}(Z,\mathcal{F}(-c))\xrightarrow{\sim} \H^r_Z(X,\mathcal{F}) $$
		for all $r\geq 0$. 
	\end{thrm}
	
	By applying Theorem \ref{thm:exactAxiom} for the sheaf $\mu_n$, and by replacing $\H^r_Z(X,\mu_n)$ with $\H^{r-2c}(Z,\mu_n(-c))$ using Theorem \ref{thm:GabberPurity}, we obtain 
	the following version of the Gysin sequence for the pair $(X,Z)$ (see e.g.\ \cite[(2)]{BrightBad}).
	
	\begin{lem}[Gysin sequence] \label{lem:Gysin}
		Let $X$ be a smooth variety over a field $k$ of characteristic $0$ and $Z\subset X$ a closed reduced smooth subscheme of codimension $c$. Then there is a long exact sequence for the pair $(X,Z)$,
		\begin{equation}
			\begin{split}
				0 & \rightarrow \Ho^{2c-1}(X, \mu_n)   \rightarrow \Ho^{2c-1}(X\setminus Z, \mu_n)  \rightarrow \Ho^{0}(Z, \mu_n(-c))  \rightarrow
				\\ & \rightarrow \Ho^{2c}(X, \mu_n)    \rightarrow \Ho^{2c}(X\setminus Z, \mu_n)  \rightarrow \Ho^{1}(Z,\mu_n (-c)) \rightarrow \cdot\cdot\cdot
			\end{split}
		\end{equation}
	\end{lem}

	We also use the following version which allows for $Z$ to be singular (see e.g.\ \cite[Corollary 2.4]{BrightBad}).
	\begin{lem}[Semi-purity in codimension $1$]\label{lem:GysinSing}
		Let $X$ be a smooth variety over a field $k$ of characteristic $0$ and $Z\subset X$ be a reduced, closed subscheme everywhere of codimension 1. Let $S$ be the singular locus of $Z$ and write $X^\circ$ for $X\setminus S$ and $Z^\circ$ for $Z\setminus S$.  Then there is a long exact sequence
		\begin{equation}
			\begin{split}
				0 & \rightarrow \Ho^1(X, \mu_n)   \rightarrow \Ho^1(X\setminus Z, \mu_n)  \rightarrow \Ho^0(Z^\circ, \z/n)  \rightarrow
				\\ & \rightarrow \Ho^2(X, \mu_n)    \rightarrow \Ho^2(X\setminus Z, \mu_n)  \rightarrow \Ho^1(Z^\circ,\z/n) \rightarrow 
				\\ &  \rightarrow \Ho^3(X^\circ, \mu_n)    \rightarrow \Ho^3(X\setminus Z, \mu_n)  \cdot\cdot\cdot
			\end{split}
		\end{equation}
	\end{lem}

	To prove Theorem \ref{thm:Main}, we require the following cohomological
	vanishing results for smooth surfaces.
	\begin{lem}\label{lem:H4Utrivial}
		Let $X$ be a smooth proper geometrically integral surface over a field $k$ of characteristic $0$ and let $U\subset X$ be a non-proper open subscheme. Then the following statements hold.
		\begin{enumerate}[itemsep=2pt,parsep=2pt,topsep=4pt]
			\item $\h^4(\overline{U}, \z/n(c))= 0$ for all $c$.
			\item If $\pic(\overline{X})$ is torsion-free, then $\h^3(\overline{X}, \z/n (c))=0$ for all $c$.
		\end{enumerate}
	\end{lem}
	\begin{proof}
		First, note that by the Lefschetz principle (see e.g.\ \cite{LeftPrin}) it suffices to prove the result for $\overline{k}=\mathbb{C}$.
		
		(1) Since $U$ is open in $X$, we have that $U$ is smooth over $k$. As $\z/n$ is a finite abelian group, we may apply the Artin comparison theorem (see e.g.\ \cite[Theorem 21.1]{MilneEC}) to obtain
		$$\h^4_{\et}(\overline{U}, \z/n(c))= \h^4_{\sing}(\overline{U}_{\cxan}, \z/n(c))$$
		where $\overline{U}_{\cxan}$ is $\overline{U}$ regarded as a complex manifold. Since $\overline{U}_{\cxan}$ is not compact, we have $\h^4_{\sing}(\overline{U}_{\cxan}, \z/n(c))=0$, as desired.
		
		(2) By Poincar{\'e} duality (see e.g.\ \cite[Theorem 24.1]{MilneEC}), we have $$\h^3(\overline{X}, \z/n (c))\cong \h^1(\overline{X}, \mu_n)$$ as abelian groups. From the Kummer exact sequence
		$$0\rightarrow \mu_n \rightarrow \gm \xrightarrow{n} \gm \rightarrow 0,$$
		we obtain the exact sequence
		$$0 \rightarrow \Ho^1(\overline{X},\mu_n) \rightarrow \pic(\overline{X})  \xrightarrow{n} \pic(\overline{X}).$$
		Since $\pic(\overline{X})$ is torsion-free, we deduce that $\h^1(\overline{X}, \mu_n)=0$ as desired.
	\end{proof}
	
	The following result allows us to truncate the Gysin
	sequence associated with the removal of a finite subscheme from a
	smooth proper surface.
	\begin{lem}\label{lem:SminPoints}
		Let $X$ be a smooth proper geometrically integral surface over a field $k$ of characteristic $0$ such that $\pic(\overline{X})$ is torsion-free. Let $P$ be a non-empty reduced finite $k$-subscheme of $X$. Then the Gysin sequence induces the exact sequence
		\begin{equation*}
			0 \rightarrow \Ho^3(\overline{X}\setminus P,\mu_n) \rightarrow \h^0(\overline{P},\z/n(-1))  \rightarrow \h^4(\overline{X},\mu_n) \rightarrow 0,
		\end{equation*}
		where we denote the map $\Ho^3(\overline{X}\setminus P,\mu_n) \rightarrow \h^0(\overline{P},\z/n(-1))$ by $\psi$.
		
		If moreover $P\in X(k)$, then $\Ho^3(\overline{X}\setminus P,\mu_n)=0$ and the map
		\begin{equation}\label{lem:H0toH4}
			\h^0(P,\z/n(-1)) \xrightarrow{}  { \h^4(\overline{X},\mu_n)}
		\end{equation}
		is an isomorphism.
	\end{lem}
	\begin{proof}
		Applying the Gysin sequence to $(X,P)$ we obtain the exact sequence
		$$ \h^3(\overline{X}, \mu_n) \rightarrow \h^3(\overline{X}\setminus P,\mu_n) \xrightarrow{\psi} \h^0(P,\z/n(-1))  \rightarrow \h^4(\overline{X},\mu_n) \rightarrow \h^4(\overline{X}\setminus P,\mu_n).$$
		By Lemma \ref{lem:H4Utrivial}, we have $\h^4(\overline{X}\setminus P,\mu_n)= 0$ and $\h^3(\overline{X},\mu_n)=0$, which proves the first statement. If $P\in X(k)$, then, by Poincar{\'e} duality, $\h^4(\overline{X},\mu_n)$ and  $\h^0(P,\z/n(-1))$ have the same cardinality, which completes the proof.
	\end{proof}

	We now apply Lemma \ref{lem:SminPoints} to two $k$-points
	lying on a line in a smooth cubic surface. The following lemma
	describes explicitly the resulting Gysin maps and their compatibility
	with the residue maps on the punctured line.
	\begin{lem}\label{lem:CommDiag}
		Let $X$ be a smooth cubic surface over a field $k$ of characteristic $0$. Let $\ell\cong \mathbb{P}^1\subset X$ and $Q_1,Q_2 \in \ell(k)$ be distinct. Set $X':=X\setminus \{Q_1,Q_2\}$, $\ell':=\ell\setminus \{Q_1,Q_2\}$, and $U:=X\setminus \ell$.
		\begin{enumerate}[label=(\roman*),itemsep=2pt,parsep=2pt,topsep=4pt]
			\item There exists a commutative diagram of $\Gamma_k$-modules with exact rows
			% https://q.uiver.app/#q=WzAsNyxbMCwwLCJcXGheMShcXGJhcntcXGVsbCd9LFxcei9uXFx6KSJdLFsxLDAsIlxcaF4zKFxcYmFye1gnfSxcXG11X24pIl0sWzIsMCwiXFxoXjAoXFxiYXJ7UH1fMVxcY3VwIFxcYmFye1B9XzIsXFx6L25cXHooLTEpKSJdLFswLDEsIlxcei9uXFx6KC0xKSJdLFsyLDEsIlxcei9uXFx6KC0xKVxcb3BsdXMgXFx6L25cXHooLTEpIl0sWzMsMCwiIFxcaF40KFxcYmFye1h9LFxcbXVfbikiXSxbMywxLCJcXHovblxceigtMSkiXSxbMCwxLCJcXHZhcnBoaSJdLFsxLDIsIlxccHNpIiwwLHsic3R5bGUiOnsidGFpbCI6eyJuYW1lIjoiaG9vayIsInNpZGUiOiJ0b3AifX19XSxbMCwzXSxbMyw0LCJcXHBhcnRpYWwiXSxbMiw0XSxbMiw1XSxbNSw2XSxbNCw2LCJkIl1d
			\begin{equation}\label{diag:lem9}
				\begin{tikzcd}[row sep=1.7em, column sep=1.7em]
					{\h^3(\overline{X'},\mu_n)} & {\oplus_{i=1}^2 \h^0(\overline{Q}_i,\z/n(-1))} & { \h^4(\overline{X},\mu_n)} \\
					{\z/n(-1)} & {\z/n(-1)\oplus \z/n(-1)} & {\z/n(-1)}
					\arrow["\rotatebox{90}{$\sim$}",from=1-1, to=2-1]
					\arrow["\psi", hook, from=1-1, to=1-2]
					\arrow[from=1-2, to=1-3]
					\arrow["\rotatebox{90}{$\sim$}",from=1-2, to=2-2]
					\arrow["\rotatebox{90}{$\sim$}",from=1-3, to=2-3]
					\arrow["\partial", from=2-1, to=2-2]
					\arrow["d", from=2-2, to=2-3]
				\end{tikzcd}
			\end{equation}
			such that $d(a,b)=a+b$ for $a,b \in \z/n(-1)$, and $\partial(a)=(a,-a)$ for ${a\in \z/n(-1)}$.
			
			\item The map $\varphi:\h^1(\overline{\ell'},\z/n) \rightarrow \h^3(\overline{X'},\mu_n)$ induced by the Gysin sequence is an isomorphism producing a commutative diagram
			% https://q.uiver.app/#q=WzAsMyxbMCwwLCJcXGheMShcXGJhcntcXGVsbCd9LFxcei9uKSJdLFsxLDEsIlxcaF4zKFxcYmFye1gnfSxcXG11X24pIl0sWzIsMCwiXFxvcGx1c197aT0xfV4yIFxcaF4wKFxcYmFye1F9X2ksXFx6L25cXHooLTEpKSJdLFswLDEsIlxcdmFycGhpIl0sWzAsMiwiXFxvcGx1c1xccGFydGlhbF97UV9pfSJdLFsxLDIsIlxccHNpIl1d
			\[\begin{tikzcd}
				{\h^1(\overline{\ell'},\z/n)} && {\oplus_{i=1}^2 \h^0(\overline{Q}_i,\z/n(-1))} \\
				& {\h^3(\overline{X'},\mu_n)}
				\arrow["{\oplus\partial_{Q_i}}", from=1-1, to=1-3]
				\arrow["\varphi", from=1-1, to=2-2]
				\arrow["\psi", from=2-2, to=1-3]
			\end{tikzcd}\]
			
			up to a sign, where $\partial_{Q_i}$ are the (Gysin) residue maps. 
		\end{enumerate}
		
	\end{lem}
	\begin{proof}
		(i) By taking $P=Q_1\sqcup Q_2$ in Lemma \ref{lem:SminPoints}, we obtain the exact sequence
		\begin{equation}\label{eqn:G1}
			0 \rightarrow \h^3(\overline{X'},\mu_n) \xrightarrow{\psi} \h^0(\overline{Q_1},\z/n(-1))\oplus \h^0(\overline{Q_2},\z/n(-1)) \rightarrow \h^4(\overline{X},\mu_n) \rightarrow 0.
		\end{equation}
		Applying the Gysin sequence (Lemma \ref{lem:Gysin}) to $(X',\ell')$ we obtain the exact sequence
		\begin{equation}\label{eqn:G2}
			\h^2(\overline{X'},\mu_n) \rightarrow \h^2(\overline{U},\mu_n) \rightarrow \h^1(\overline{\ell'},\z/n) \xrightarrow{\varphi} \h^3(\overline{X'},\mu_n) \rightarrow \h^3(\overline{U},\mu_n),
		\end{equation}
		and to $(X,\ell)$ we get the exact sequence
		\begin{equation}\label{eqn:G21}
			\h^2(\overline{X},\mu_n) \rightarrow \h^2(\overline{U},\mu_n) \rightarrow  \h^1(\overline{\ell},\z/n) \cong 0.
		\end{equation} 
		By \cite[Corollary 2.3]{BrightBad}, the natural map $\h^2(\overline{X},\mu_n)\rightarrow \h^2(\overline{X'},\mu_n)$ is an isomorphism. Thus, from \eqref{eqn:G2} and  \eqref{eqn:G21}, we obtain an exact sequence 
		$$0 \rightarrow \h^1(\overline{\ell'},\z/n) \xrightarrow{\varphi} \h^3(\overline{X'},\mu_n) \rightarrow \h^3(\overline{U},\mu_n).$$
		From \eqref{eqn:G1} we see that $\h^3(\overline{X'},\mu_n)$ has cardinality $n$ which is equal to that of $\h^1(\overline{\ell'},\mu_n)$. Therefore, we obtain an isomorphism
		$$\varphi:\h^1(\overline{\ell'},\z/n) \xrightarrow{\sim} \h^3(\overline{X'},\mu_n).$$
		Now, the vertical maps in \eqref{diag:lem9} come from the maps $${ \h^1(\overline{\ell'}, \z/n)\xrightarrow{\sim} \h^1(\gm, \z/n)\xrightarrow{\sim} \z/n(-1)},$$ ${\h^0(\overline{Q_i}, \z/n(-1))\xrightarrow{\sim} \z/n(-1)}$, and $\h^4(\overline{X}, \mu_n) \xrightarrow{\sim} \z/n(-1)$ (where the last one follows by Poincaré duality). Then there exist maps $\partial$ and $d$ producing the commutative diagram \eqref{diag:lem9}. It remains to check that we may choose $d$ and $\partial$ to be the claimed maps. 
		
		Since the Gysin sequence (Lemma \ref{lem:Gysin}) was obtained by applying purity (Theorem \ref{thm:GabberPurity}), the map $d$ is the composition of
		$$\oplus_{i=1}^2 \h^0(\overline{Q}_i,\z/n(-1)) \xrightarrow{} \oplus_{i=1}^2 \h^4_{\overline{Q}_i}(\overline{X},\mu_n) \rightarrow \h^4_{\overline{Q}_1 \cup \overline{Q}_2}(\overline{X},\mu_n) \rightarrow  { \h^4(\overline{X},\mu_n)}$$
		where the first map is the purity isomorphism, the second is induced by the sum of global sections, and the third is induced by the inclusion of global sections. This proves that $d$ is the claimed map. Note that by the exactness of the bottom row in \eqref{diag:lem9}, by choosing the left vertical map to be the composition of $\psi$ with the projection to the first component of the middle vertical map, we obtain the desired map $\partial$.
		
		(ii) Let $Q:= Q_1 \sqcup Q_2$. It suffices to prove the commutativity of the diagram
	\begin{equation}\label{diag:6obj}
		\begin{tikzcd}[column sep=large,row sep=large]
			\Ho^1(\ell',\mathbb Z/n)
			\arrow[r]
			\arrow[d,"\varphi"']
			&
			\Ho^2_Q(\ell,\mathbb Z/n)
			\arrow[d]
			&
			\Ho^0(Q,\mathbb Z/n(-1))
			\arrow[l,"\sim","\mathrm{pur}"']
			\arrow[d,equal]
			\\
			\Ho^3(X',\mu_n)
			\arrow[r]
			&
			\Ho^4_Q(X,\mu_n)
			&
			\Ho^0(Q,\mathbb Z/n(-1))
			\arrow[l,"\sim","\mathrm{pur}"'],
		\end{tikzcd}
	\end{equation}
	where the compositions of the horizontal maps are the Gysin residue maps.
	Write $v$ and $i$ for the closed immersions
	$Q\hookrightarrow\ell$ and $\ell\hookrightarrow X$,
	respectively. For a closed immersion $f:Z\hookrightarrow Y$,
	we write $f^!$ for the right adjoint of $f_*$,
	and $R^qf^!$ and $\mathbf R f^!$ for its $q$-th
	and total right-derived functors, respectively
	(see \cite[p.~56]{CTSK}).

	We first show that the right-hand square is commutative.
	By the transitivity of purity
	\cite[Exp.\ XIX, Theorem 3.4(iii)]{SGA4}, the following diagram of
	sheaves on $Q$ is commutative:
	\begin{equation}\label{diag:transPurity}
		\begin{tikzcd}[column sep=large,row sep=large]
			(\mathbb Z/n(-1))
			\arrow[r]
			\arrow[d]
			&
			R^2v^!(\mathbb Z/n)
			\arrow[d]
			\\
			R^4(i\circ v)^!\mu_n
			&
			R^2v^!R^2i^!\mu_n
			\arrow[l]
		\end{tikzcd}
	\end{equation}
	Since $R^q v^!(\z/n)=0$ for $q\neq 2$ and $R^q(i\circ v)^!\mu_n =0$ for $q\neq 4$ (see e.g.\ \cite[Theorem 2.3.1]{CTSK}),
	the spectral sequences (see e.g.\ \cite[(2.10), p.\ 57]{CTSK})
	$$
	E_2^{p,q}
	=
	\Ho^p\bigl(Q,R^q v^!(\z/n) \bigr)
	\Longrightarrow
	\Ho_Q^{p+q}(\ell,\z/n)
	$$ 
	and 
	$$
	E_2^{p,q}
	=
	\Ho^p\bigl(Q, R^q(i\circ v)^!\mu_n\bigr)
	\Longrightarrow
	\Ho_Q^{p+q}(X,\mu_n)
	$$ 
	induce the canonical edge
	isomorphisms
	$$
	\Ho^0\bigl(Q,R^2 v^!(\z/n) \bigr)
	\xrightarrow{\sim}
	\Ho_Q^{2}(\ell,(\z/n))\quad  \text{ and }\quad \Ho^0\bigl(Q, R^4(i\circ v)^!\mu_n\bigr)
	\xrightarrow{\sim}
	\Ho_Q^{4}(X,\mu_n).
	$$
	Therefore, by applying $\Ho^0(Q,-)$ to \eqref{diag:transPurity}, we obtain the commutativity of the right-hand square of \eqref{diag:6obj}.
	
	We now prove the commutativity of the left-hand square of \eqref{diag:6obj}. By absolute purity \cite[Theorem 2.3.1]{CTSK}, we have a canonical isomorphism
	$$
	\theta \colon (\z/n)_{\overline{\ell}}[-2]
	\xrightarrow{\sim} \mathbf{R}i^!\mu_n.
	$$
	For any complex of sheaves $K$ on $\overline{\ell}$, there is a
	distinguished triangle
	\[
	R\Gamma_Q(\overline{\ell},K)
	\longrightarrow R\Gamma(\overline{\ell},K)
	\longrightarrow R\Gamma(\overline{\ell}',K|_{\overline{\ell}'})
	\longrightarrow R\Gamma_Q(\overline{\ell},K)[1],
	\]
	which is functorial in $K$ (see e.g.\ \cite[Lemma 20.34.5]{Stacks}).
	Applying this functoriality to $\theta$, and then taking cohomology,
	gives a commutative diagram
	\begin{equation}
	\begin{tikzcd}\label{diag:functOfTheta}
		\h^1(\overline{\ell},\mathbb{Z}/n)
		\arrow[r]
		\arrow[d,"\rotatebox{90}{$\sim$}"]
		&
		\h^1(\overline{\ell}',\mathbb{Z}/n)
		\arrow[r]
		\arrow[d,"\rotatebox{90}{$\sim$}"]
		&
		\h^2_Q(\overline{\ell},\mathbb{Z}/n)
		\arrow[d,"\rotatebox{90}{$\sim$}"]
		\\
		\h^3(\overline{\ell},\mathbf{R}i^!\mu_n)
		\arrow[r]
		&
		\h^3(\overline{\ell}',\mathbf{R}i^!\mu_n)
		\arrow[r]
		&
		\h^4_Q(\overline{\ell},\mathbf{R}i^!\mu_n).
	\end{tikzcd}
	\end{equation}
	By the definition of $i^!$ and cohomology with supports, the bottom row
	identifies canonically with the first row in the following diagram
	\begin{equation}
	\begin{tikzcd}\label{diag:ForgetSupport}
		\h^3_{\overline{\ell}}(\overline{X},\mu_n)
		\arrow[r]
		\arrow[d]
		&
		\h^3_{\overline{\ell}'}(\overline{X}',\mu_n)
		\arrow[r]
		\arrow[d]
		&
		\h^4_Q(\overline{X},\mu_n)
		\arrow[d,equal]
		\\
		\h^3(\overline{X},\mu_n)
		\arrow[r]
		&
		\h^3(\overline{X}',\mu_n)
		\arrow[r]
		&
		\h^4_Q(\overline{X},\mu_n).
	\end{tikzcd}
	\end{equation}
	Let $\mu_n\to I^\bullet$ be an injective resolution
	on $\overline X$. We have a commutative diagram of
	short exact sequences of complexes
	\[
	\begin{tikzcd}[column sep=small]
		0 \arrow[r]
		&
		\Gamma_Q(\overline X,I^\bullet)
		\arrow[r] \arrow[d,equal]
		&
		\Gamma_{\overline\ell}(\overline X,I^\bullet)
		\arrow[r] \arrow[d,hook]
		&
		\Gamma_{\overline{\ell'}}(\overline X',I^\bullet)
		\arrow[r] \arrow[d,hook]
		&
		0
		\\
		0 \arrow[r]
		&
		\Gamma_Q(\overline X,I^\bullet)
		\arrow[r]
		&
		\Gamma(\overline X,I^\bullet)
		\arrow[r]
		&
		\Gamma(\overline X',I^\bullet)
		\arrow[r]
		&
		0.
	\end{tikzcd}
	\]
	By the naturality of the connecting homomorphisms
	associated with these short exact sequences,
	the diagram \eqref{diag:ForgetSupport} commutes. 
	By combining \eqref{diag:functOfTheta} and \eqref{diag:ForgetSupport}, and observing that the induced map $\h^1(\overline{\ell'},\z/n) \xrightarrow{\sim} \h^3(\overline{X'},\mu_n)$ by the vertical maps coincides with $\varphi$, we conclude that the left-hand square of \eqref{diag:6obj} is commutative, which completes the proof. 
	\end{proof}

	\subsection{Proof of Theorem \ref{thm:Main}}\label{subsec:BrUbar}
	First, from the Kummer exact sequence
	$$0\rightarrow \mu_n \rightarrow \gm \xrightarrow{n} \gm \rightarrow 0,$$
	we obtain the following commutative diagram with exact rows:
	\begin{equation} \label{diag:kummer-brStoBrU}
		\begin{tikzcd}  
			0 & {(\Pic \overline{X})/n} & {\Ho^2(\overline{X},\mu_n)} & {(\Br \overline{X})[n]} & 0 \\
			0 & {(\Pic \overline{U})/n} & {\Ho^2(\overline{U},\mu_n)} & {(\Br \overline{U})[n]} & 0
			\arrow[from=1-1, to=1-2]
			\arrow[from=1-2, to=1-3]
			\arrow[from=1-2, to=2-2]
			\arrow[from=2-1, to=2-2]
			\arrow[from=2-2, to=2-3]
			\arrow[from=2-3, to=2-4]
			\arrow[from=1-3, to=2-3]
			\arrow[from=1-4, to=2-4]
			\arrow[from=1-4, to=1-5]
			\arrow[from=2-4, to=2-5]
			\arrow[from=1-3, to=1-4]
		\end{tikzcd}
	\end{equation}
	where we write $\Psi_n$ for the map $\Ho^2(\overline{X},\mu_n)\rightarrow \Ho^2(\overline{U},\mu_n).$
	
	\begin{lem}\label{lem:BrCoker}
		We have an isomorphism
		$$\br(\overline{U})[n]\cong \coker(\Psi_n). $$
	\end{lem}
	\begin{proof}
		The left vertical map in \eqref{diag:kummer-brStoBrU} is surjective as $X$ is smooth. Since $X$ is a geometrically rational surface, we have $\br(\overline{X})=0$. The result now follows by the Snake lemma.
	\end{proof}
	
	In light of Lemma \ref{lem:BrCoker}, it suffices to compute $\text{coker}(\Psi_n)$, for all $n$, for all the different possible singular hyperplane sections.
	
	\medskip
	\noindent\textbf{Case 1: A smooth conic and a line both defined over $k$.}
	\par\smallskip

	\begin{prop} \label{prop:case1}
		Suppose that $H$ is the union of a line $\ell$ and a smooth conic $C$ both defined over $k$. If $\ell$ meets $C$ at rational points, then as $\Gamma_k$-modules we have 
		\[  \br \overline{U} \cong 
		\begin{cases}
			0, & \text{if $\ell$ is tangent to C,} \\
			\q/\z(-1), & \text{if $|\ell(k)\cap C(k)|=2$.} \\
		\end{cases} 
		\]
		If $\ell$ meets $C$ at a reduced degree $2$ point with residue field $L=k(\sqrt{d})$, $d\notin k^2$, then as $\Gamma_k$-modules
		$$(\br \overline{U})[n] \cong M_d/nM_d(-1),$$
		for all $n\in \z_{>0}$, where $M_d=(\Ind_{k(\sqrt{d})/k} \z) / \z$.
	\end{prop}
	
	\begin{proof}
		Let $U_1=X\setminus C$. Applying the Gysin sequence to $(X,C)$ we get the exact sequence 
		\begin{equation}\label{eqn:S,U1}
			\Ho^2(\overline{X}, \mu_n)\rightarrow \Ho^2(\overline{U_1},\mu_n )\rightarrow \Ho^1(\overline{C}, \z/n) \rightarrow \cdots
		\end{equation}
		Since $\overline{C}\cong \mathbb{P}^1$ we have $\Ho^1(\overline{C}, \z/n)=0$ so $\Ho^2(\overline{X}, \mu_n)\rightarrow \Ho^2(\overline{U_1},\mu_n )$ is surjective. The exact sequence \eqref{eqn:S,U1} continues to a long exact sequence
		$$\Ho^3(\overline{X}, \mu_n)\rightarrow \Ho^3(\overline{U_1},\mu_n )\rightarrow \Ho^2(\overline{C}, \z/n)\rightarrow \Ho^4(\overline{X}, \mu_n)\rightarrow \Ho^4(\overline{U_1},\mu_n ). $$
		By Lemma \ref{lem:H4Utrivial}, we have $\Ho^4(\overline{U_1},\mu_n )=0$ and $\Ho^3(\overline{X}, \mu_n)=0$.
		By Poincaré duality, $\Ho^2(\overline{C}, \z/n)$ and $\Ho^4(\overline{X}, \mu_n)$ are isomorphic to $\z/n$ as abelian groups, which implies that 
		$\Ho^2(\overline{C}, \z/n)\rightarrow \Ho^4(\overline{X}, \mu_n)$ is an isomorphism and that
		$\Ho^3(\overline{U_1},\mu_n )=0$.
		Applying the Gysin sequence to $(U_1, U_1\cap \ell)$ we get 
		$$\Ho^2(\overline{U_1}, \mu_n)\rightarrow \Ho^2(\overline{U},\mu_n )\rightarrow \Ho^1(\overline{U_1}\cap \overline{\ell}, \z/n)\rightarrow \Ho^3(\overline{U_1},\mu_n )=0.$$
		Since $\Ho^2(\overline{X}, \mu_n)\rightarrow \Ho^2(\overline{U_1},\mu_n )$ is surjective, we have
		$$(\Br \overline{U})[n]  \cong \Ho^1(\overline{U_1}\cap \overline{\ell}, \z/n).$$
		
		If $\ell$ is tangent to $C$, then $U_1\cap \ell \cong \mathbb{A}^1$, so that $\Ho^1(\overline{U_1}\cap \overline{\ell}, \z/n)=0$. If $\ell$ intersects $C$ in two distinct $k$-points, then $U_1\cap \ell \cong \mathbb{G}_m$. By Kummer theory, $\Ho^1(\mathbb{G}_{m,\overline{k}},\mu_n)=\z/n$ so that $\Ho^1(\mathbb{G}_{m,\overline{k}},\z/n)=\z/n(-1)$.
		
		If $\ell$ meets $C$ at a degree $2$ point $P$ with residue field $L=k(\sqrt{d})$, then $$\ell\setminus P\cong \Res_{L/k} \mathbb{G}_{m,L}/\mathbb{G}_{m,k} =:T,$$ is a $1$-dimensional torus over $k$, where $\Res_{L/k}$ is the Weil restriction. Let ${G:=\text{Gal}(L/k)\cong \z/2}$. We have an exact sequence
		\begin{equation}\label{eqn:Tsplits}
			1\rightarrow \overline{k}^{\times}\rightarrow \Ho^0(\overline{T}, \mathbb{G}_{m})\rightarrow U(\overline{T})\rightarrow 1,
		\end{equation}
		where $U(\overline{T})=\Hom_{k-\text{groups}}(\overline{T},\mathbb{G}_{m,\overline{k}})$ which is isomorphic to $\z$ as an abelian group. Note that the action of $G$ on $U(\overline{T})$ translates to $\z$ as the action induced by ${G\rightarrow \text{Aut}(\z)\cong \{\pm 1\}}$. More concretely, we have 
		$$U(\overline{T})\cong (\Ind_{L/k} \z) / \z=M_d$$ as $\Gamma_k$-modules, where the $\z \hookrightarrow \Ind_{L/k} \z$ is the diagonal embedding.
		Since $T$ has a rational point, we see that \eqref{eqn:Tsplits} splits, i.e.\
		$$\Ho^0(\overline{T}, \mathbb{G}_{m})\cong \overline{k}^{\times} \oplus  U(\overline{T}).$$ Hence, from the Kummer sequence, and using that $\Ho^1(\overline{T}, \mathbb{G}_{m})=0$, we deduce that $\Ho^1(\overline{T}, \mu_n)\cong M_d/nM_d$. Therefore,
		$(\Br \overline{U})[n] \cong \Ho^1(\overline{T}, \z/n)\cong M_d/nM_d(-1).$
	\end{proof}

	\medskip
	\noindent\textbf{Case $2$: a geometrically irreducible singular cubic curve over $k$}
	\par\smallskip
	
	\begin{prop}\label{prop:case2}
		Suppose that $H$ is a geometrically irreducible singular cubic curve $C$. Then as $\Gamma_k$-modules we have 
		\[  \br \overline{U} \cong 
		\begin{cases}
			0, & \text{if $C$ is cuspidal,} \\
			\q/\z(-1), & \text{if $C$ is nodal split multiplicative.} \\
		\end{cases} 
		\]
		If $C$ is nodal non-split multiplicative splitting over $L=k(\sqrt{d})$, $d\notin k^2$, then
		$$(\br\overline{U})[n] \cong M_d/nM_d(-1)$$ as $\Gamma_k$-modules for all $n$,
		where $M_d:=(\Ind_{k(\sqrt{d})/k} \z) / \z$.
	\end{prop}
	
	\begin{proof}
		First, note that the singular point $P$ of $C$ has to be a $k$-point.
		Set $X'=X\setminus \{P\}$ and $C'=C\setminus \{P\}$. By Lemma \ref{lem:SminPoints}, we have that $\Ho^3(\overline{X'},\mu_n )=0$.
		Thus, by applying Lemma \ref{lem:GysinSing} to $(X',C')$ we get the exact sequence 
		$$\Ho^2(\overline{X}, \mu_n)\rightarrow \Ho^2(\overline{U},\mu_n )\rightarrow \Ho^1(\overline{C'}, \z/n)\rightarrow 0.$$
		If $C$ is cuspidal, then $\Ho^1(\overline{C'}, \z/n)=\Ho^1(\mathbb{A}^1, \z/n)=0$.
		If $C$ is nodal split multiplicative, then $\Ho^1(\overline{C'}, \z/n)=\Ho^1(\mathbb{G}_{m,\overline{k}}, \z/n)=\z/n(-1)$. Finally, if $C$ is nodal non-split multiplicative, then $C' \cong \Res_{L/k} \mathbb{G}_{m,L}/\mathbb{G}_{m,k}$ so that
		${\Ho^1(\overline{C'}, \z/n) \cong M_d/nM_d(-1)}.$
	\end{proof}

	\medskip
	\noindent\textbf{Case $3$: three geometric lines with a Galois action}
	\par\smallskip

	Suppose that $C:=X\cap H$ is geometrically the union of three lines and let $L$ be the minimal extension over which they are defined. We assume that the lines do not intersect in an Eckardt point, as in this case $\br \overline{U} = 0$ (see \cite[Proposition 2.4]{MarkoffSurf}). Let $P$ be the singular locus of $C$ and set $X':=X\setminus P$, $C':=C\setminus P$. We can write $P_L=\cup_{i=1}^3 P_i$, where each $P_i\in X_L(L)$, and $C_L:=\cup_{i=1}^3 \ell_i$, where each $\ell_i\cong \mathbb{P}^1_L$ and such that $\cup_{j\neq i} P_j \subset \ell_i$ for each $i$.
	Set $\ell_i' := \ell_i \setminus P$, so that
	$C_L' = \bigcup_{i=1}^3 \ell_i'$.
	Write $P=\operatorname{Spec}(E)$, where $E$ is a degree-$3$
	étale $k$-algebra. We list the possibilities for $E$:
	\begin{enumerate}[label=(\roman*),itemsep=2pt,parsep=2pt,topsep=4pt]
		\item $E=\prod_{i=1}^3 k$. In this case, $\br \overline{U} = \q/\z(-1)$ (see \cite[Proposition 2.4]{MarkoffSurf});
		\item $E=L\times k$, where $L:=k(\sqrt{d})$, $d\notin k^2$ and $\gal(L/k)\cong\z/2$;
		\item $E=L$, where $L$ is a cubic Galois extension of $k$ and $\gal(L/k)\cong\z/3$;
		\item $E$ is a non-Galois cubic extension of $k$ with
		Galois closure $L=E(\sqrt d)$, for some $d\notin k^2$.
		In this case, $\Gal(L/k)\cong S_3$.
	\end{enumerate}
	
	We compute $\operatorname{Br}(\overline U)$ as a
	$\Gamma_k$-module in all four cases.
	\begin{prop} \label{prop:case3}
		As $\Gamma_k$-modules, we have 
		\[ \br(\overline{U}) \cong  
		\begin{cases}
			\q/\z(-1), & \text{if $L=k$ or if $\gal(L/k)\cong\z/3$,} \\
			\varinjlim \left( {M_d}/{nM_d}(-1) \right), & \text{if $\gal({L}/k)\cong S_3$ or $\gal(L/k)\cong \z/2$,} \\
		\end{cases} 
		\]	
		where $M_d:=(\Ind_{k(\sqrt{d})/k} \z) / \z$.
	\end{prop}
	\begin{proof}
		By Lemma \ref{lem:SminPoints}, we have an exact sequence of $\Gamma_k$-modules
		\begin{equation}\label{eq:prop3.8.1}
			0 \rightarrow \h^3(\overline{X'},\mu_n) \xrightarrow{\psi} \h^0(\overline{P},\z/n(-1)) \rightarrow \h^4(\overline{X},\mu_n) \rightarrow 0
		\end{equation}
		The Gysin sequence applied to  $(X',C')$ gives a map $$\varphi:\h^1(\overline{C'},\mathbb Z/n)
		\to\h^3(\overline{X'},\mu_n)$$ which is part of the long exact
		sequence of $\Gamma_k$-modules:
		\begin{equation}\label{eq:prop3.8.2}
			\h^2(\overline{X'},\mu_n) \rightarrow \h^2(\overline{U},\mu_n) \rightarrow \h^1(\overline{C'},\z/n) \xrightarrow{\varphi} \h^3(\overline{X'},\mu_n) \rightarrow \h^3(\overline{U},\mu_n).
		\end{equation}
		We wish to compute $\ker(\varphi)$ as a $\Gamma_k$-module.  We first compute it as a $\Gamma_L$-module.  Note that, as $\Gamma_L$-modules, we have $$\h^1(\overline{C'},\z/n)=\oplus_{i=1}^3 \h^1(\overline{\ell'_i},\z/n) \> \text{ and } \> \h^0(\overline{P},\z/n(-1))=\oplus_{i=1}^3 \h^0(\overline{P}_i,\z/n(-1)).$$

		\medskip
		\textbf{Step 1: explicit description of the map $\varphi$.}
		\par\smallskip
		
		For all $j\in \{1,2,3\}$, by applying Theorem \ref{thm:exactAxiom} to the pairs $(X, P)$ and ${(X,\cup_{i\neq j} P_i)}$, we obtain the following commutative diagram with exact rows:
		\begin{equation}\label{diag:Exact1}
			\begin{tikzcd}
				{\h^3(\overline{X'},\mu_n)} & {\bigoplus_{i=1}^3 \h^4_{\overline{P}_i}(\overline{X},\mu_n)} & {\h^4(\overline{X},\mu_n)} \\
				{\h^3(\overline{X}\setminus \cup_{i\neq j} \overline{P}_i,\mu_n)} & {\bigoplus_{i\neq j} \h^4_{\overline{P}_i}(\overline{X},\mu_n)} & {\h^4(\overline{X},\mu_n)}
				\arrow["\begin{array}{c} \psi \end{array}", from=1-1, to=1-2]
				\arrow[from=1-2, to=1-3]
				\arrow[from=2-1, to=1-1]
				\arrow[from=2-1, to=2-2]
				\arrow[from=2-2, to=1-2]
				\arrow[from=2-2, to=2-3]
				\arrow[from=2-3, to=1-3]
			\end{tikzcd}
		\end{equation}
		By applying Theorem \ref{thm:exactAxiom} to the pairs $(X', \ell_j')$ and $(X\setminus \cup_{i\neq j} P_i,  \ell'_j)$ for each $j$, and using Theorem \ref{thm:GabberPurity}, we obtain a commutative diagram
		\begin{equation}\label{diag:Exact2}
			\begin{tikzcd}[column sep=normal]
				{\h^1(\overline{\ell_j'},\z/n)}
				&
				{\h^3_{\ell_j'}(\overline{X'},\mu_n)}
				&
				{\h^3(\overline{X'},\mu_n)}
				\\
				{\h^1(\overline{\ell_j'},\z/n)}
				&
				{\h^3_{\ell_j'}(
					\overline{X}\setminus
					\bigcup_{i\neq j}\overline{P}_i,\mu_n)}
				&
				{\h^3
					(\overline{X}\setminus
					\bigcup_{i\neq j}\overline{P}_i,\mu_n),}
				\arrow["{\sim}", from=1-1, to=1-2]
				\arrow[from=1-2, to=1-3]
				\arrow["{\sim}", from=2-1, to=2-2]
				\arrow[from=2-2, to=2-3]
				\arrow[dash, shift left=.1em, from=2-1, to=1-1]
				\arrow[dash, shift right=.1em, from=2-1, to=1-1]
				\arrow[from=2-2, to=1-2]
				\arrow[from=2-3, to=1-3]
			\end{tikzcd}
		\end{equation}
		where we denote by $\varphi_j$ the composition of the maps in the top row. Observe that by the identification $\h^1(\overline{C'},\z/n)=\oplus_{i=1}^3 \h^1(\overline{\ell'_i},\z/n)$, the map $\oplus_{j=1}^3 \varphi_j$ coincides with $\varphi:\h^1(\overline{C'},\z/n) \rightarrow \h^3(\overline{X'},\mu_n)$ in \eqref{eq:prop3.8.2}.
		
		By combining the diagrams \eqref{diag:Exact1} and \eqref{diag:Exact2}, we obtain the following commutative diagram  
		
		\[\begin{tikzcd}
			{\h^1(\overline{\ell'_j},\z/n)} & {\bigoplus_{i=1}^3 \h^4_{\overline{P}_i}(\overline{X},\mu_n)} & {\h^4(\overline{X},\mu_n)} \\
			{\h^1(\overline{\ell'_j},\z/n)} & {\bigoplus_{i\neq j} \h^4_{\overline{P}_i}(\overline{X},\mu_n)} & {\h^4(\overline{X},\mu_n)}
			\arrow["{\psi\circ \varphi_j}", from=1-1, to=1-2]
			\arrow[from=1-2, to=1-3]
			\arrow[dash, shift left=.1em, from=2-1, to=1-1]
			\arrow[dash, shift right=.1em, from=2-1, to=1-1]
			\arrow[from=2-1, to=2-2]
			\arrow[from=2-2, to=1-2]
			\arrow[from=2-2, to=2-3]
			\arrow[from=2-3, to=1-3]
		\end{tikzcd}\]
		for all $j\in \{1,2,3\}$. Applying Lemma \ref{lem:CommDiag} to $\ell_j$,  for each $j$, we obtain a commutative diagram of $\Gamma_L$-modules
		\[\begin{tikzcd}
			{\h^1(\overline{\ell'}_j,\z/n)} & {\oplus_{i\neq j} \h^0( \overline{P}_i,\z/n(-1))} & { \h^4(\overline{X},\mu_n)} \\
			{\z/n(-1)} & {\bigoplus_{i\neq j} \z/n(-1)} & {\z/n(-1)}
			\arrow["{\psi\circ\varphi_j}", from=1-1, to=1-2]
			\arrow["{\rotatebox{90}{$\sim$}}", from=1-1, to=2-1]
			\arrow[from=1-2, to=1-3]
			\arrow["{\rotatebox{90}{$\sim$}}", from=1-2, to=2-2]
			\arrow["{\rotatebox{90}{$\sim$}}", from=1-3, to=2-3]
			\arrow["{\partial_j}", from=2-1, to=2-2]
			\arrow["{d_j}", from=2-2, to=2-3]
		\end{tikzcd}\]
		such that $d_j(a,b)=a+b$ for $a,b \in \z/n(-1)$ and $\partial_j(a)=(a,-a)$ for $a\in \z/n(-1)$. We choose the isomorphisms $\h^1(\overline{\ell'_i},\z/n)\cong\z/n(-1)$ so that $\partial_i(a)$ has component $a$ at $P_{i+1}$ and $-a$ at $P_{i+2}$ (indices modulo $3$).
		Therefore, by combining these diagrams for all $j$, we obtain a commutative diagram of $\Gamma_L$-modules
		\begin{equation}\label{diag:1}
			\begin{tikzcd}
				{\oplus_{i=1}^3 \h^1(\overline{\ell'}_i,\z/n)}  && {\oplus_{i=1}^3 \h^0(\overline{P}_i,\z/n(-1))}  \\
				{\bigoplus_{i=1}^3 \z/n(-1)} && {\bigoplus_{i=1}^3 \z/n(-1)} 
				\arrow["\psi\circ \varphi", from=1-1, to=1-3]
				\arrow["\rotatebox{90}{$\sim$}",from=1-1, to=2-1]
				\arrow["\rotatebox{90}{$\sim$}",from=1-3, to=2-3]
				\arrow["\partial=\sum_{i=1}^3 \partial_i", from=2-1, to=2-3]
			\end{tikzcd}
		\end{equation}
		such that the restriction of $\psi\circ \varphi$ to $\h^1(\overline{\ell'}_i,\z/n)$ is $\psi\circ \varphi_i$, and
		\begin{equation}\label{eqn:resMaps}
			{\partial(a,b,c)=(c-b,a-c,b-a)}
		\end{equation}
		for $(a,b,c)\in \bigoplus_{i=1}^3 \z/n(-1)$. Thus, since $\psi$ is injective, we have \begin{equation}\label{eqn:kerPhi}
			\ker(\varphi)\cong \ker(\partial)= \{(a,a,a)\in \oplus_{i=1}^3 \z/n(-1)\}\cong \z/n(-1)
		\end{equation} as $\Gamma_L$-modules. \vspace{4pt}
		
		\medskip
		\textbf{Step 2: $\ker(\varphi)$ as a $\Gamma_k$-module.}
		\par\smallskip
		
		So far, by step 1, we have the following commutative diagram of abelian groups 
		\begin{equation}\label{diag01}
			\begin{tikzcd}
				{ \h^1(\overline{C'},\z/n)}  && { \h^0(\overline{P},\z/n(-1))}  \\
				{\bigoplus_{i=1}^3 \z/n} && {\bigoplus_{i=1}^3 \z/n} 
				\arrow["\psi\circ \varphi", from=1-1, to=1-3]
				\arrow["\rotatebox{90}{$\sim$}",from=1-1, to=2-1]
				\arrow["\rotatebox{90}{$\sim$}",from=1-3, to=2-3]
				\arrow["\partial=\sum_{i=1}^3 \partial_i", from=2-1, to=2-3]
			\end{tikzcd}
		\end{equation}
		where the top morphism is that of $\Gamma_k$-modules.  Note that $$\h^0(\overline{P},\z/n(-1))\cong \ind_{E/k} \z/n(-1)$$ as a $\Gamma_k$-module. In other words, we have
		$$\h^0(\overline{P},\z/n(-1))\cong \bigoplus_{i=1}^3 (\z/n(-1))_{P_i}$$
		where $g\in \Gamma_k$ acts by permuting the coordinates according to their action on the associated point, followed by its action on each coordinate.
		We now determine the $\Gamma_k$-action on $\h^1(\overline{C'},\z/n)$ by using the explicit description of the map $\partial$ along with the description of the $\Gamma_k$-action on $\h^0(\overline{P},\z/n(-1))$.
		
		Observe that the $\Gamma_k$-action on the points $\{\overline{P}_i\}_{i=1}^3$ determines that on the lines $\{\overline{\ell}'_i\}_{i=1}^3$. Note that $G:=\gal({L}/k)$ is isomorphic to a subgroup of $S_3$ via the action on the indices of the points $P_i$ (and hence on the indices of the lines $\ell_i$).
		\begin{enumerate}[itemsep=2pt,parsep=2pt,topsep=4pt]
			\item Suppose that $g\in \Gamma_k$ has class in $G=\Gamma_k/\Gamma_L$ equal to the $3$-cycle $(123)$. Then by \eqref{eqn:resMaps} we have
			$$\partial(g\cdot (a,0,0))=g\cdot\partial(a,0,0)=g\cdot(0,a,-a)=( -a^g, 0 , a^g)= \partial(0,a^g, 0)$$
			where $a^g$ is the action of $g\in \Gamma_k$ on $a\in \z/n(-1)$. This shows that $${g\cdot (a,0,0)}=(0,a^g,0)+(b,b,b)$$ for some $(b,b,b)\in \ker(\varphi)$. However, since $G$ acts on $\{\ell_i\}_{i=1}^3$ by permutation, we have to have $b=0$. Similarly, we compute ${g\cdot (0,a,0)=(0,0,a^g)}$ and $g\cdot (0,0,a)= (a^g,0,0)$, which implies that $$g\cdot (a,a,a) = (a^g,a^g,a^g).$$ If $g$ has class equal to the $3$-cycle $(132)$ instead, then a similar computation shows that  $g\cdot (a,a,a) = (a^g,a^g,a^g)$ as well.
			\item Suppose that $g\in \Gamma_k$ has class in $G=\Gamma_k/\Gamma_L$ equal to the $2$-cycle $(12)$ (so that $g$ flips the line $\ell_3$ and switches the lines $\ell_1$ and $\ell_2$). Then
			$$\partial(g\cdot (a,0,0))=g\cdot \partial(a,0,0)=g\cdot(0,a,-a)=( a^g, 0 , - a^g)= \partial(0,- a^g, 0)$$
			so that $g\cdot (a,0,0)=(0,- a^g, 0)$.
			Similarly, we compute $g\cdot (0,a,0)= (-a^g,0,0) $ and $g\cdot (0,0,a)=(0,0,-a^g)$. Hence, we deduce that 
			$${g\cdot (a,a,a)= (-a^g,-a^g,-a^g)}.$$ If $g$ has class equal to the $2$-cycle $(23)$ or $(13)$, then a similar computation shows that $g\cdot (a,a,a) = (-a^g,-a^g,-a^g)$ as well.
		\end{enumerate}
		Therefore, we conclude the following:
		\begin{itemize}[itemsep=2pt,parsep=2pt,topsep=4pt]
			\item if $L$ is Galois and $\gal(L/k)\cong \z/3$, then by (1) we deduce that ${\ker(\varphi)\cong \z/n(-1)}$ as a $\Gamma_k$-module;
			\item if $L$ is Galois and $\gal(L/k)\cong \z/2$, then by (2) we deduce that ${\ker(\varphi)\cong M_d/nM_d(-1)}$ as a $\Gamma_k$-module, where $M_d:=(\Ind_{k(\sqrt{d})/k} \z) / \z$. Indeed, this follows since $g\in \Gamma_k$ has class in $G$ equal to a $2$-cycle if and only if it conjugates $\sqrt{d}$; 
			\item if  $\gal({L}/k)\cong S_3$, then by (1) and (2) we deduce that ${\ker(\varphi)\cong M_d/nM_d(-1)}$ as a $\Gamma_k$-module.
		\end{itemize}
		The result now follows by taking the direct limit.
	\end{proof}

	\section{The transcendental Brauer group}
	In this section, we use Theorem \ref{thm:Main} to bound the
	transcendental Brauer group of $U$. We then list 
	all the possibilities for the transcendental Brauer group under some condition on the field $k$. We then
	construct explicit examples proving Theorem \ref{thm:DescentToQ}.
	Finally, in \S 4.3, we discuss the integral Hasse principle for
	natural models of these examples and show that a transcendental
	class gives a Brauer--Manin obstruction to integral strong
	approximation on a suitable model.

	\subsection{Computing $\br(\overline{U})^{\Gamma_k}$}
	We start with the following lemma which computes the $\Gamma_k$-invariant submodule of the $\Gamma_k$-module $M_d/nM_d(-1)$ appearing in Theorem \ref{thm:Main}.
	\begin{lem}\label{lem:FixedByGalAction}
		Let $n$ be a prime power and $d\in k\setminus k^2$. Set $M_d:=(\Ind_{k(\sqrt{d})/k} \z) / \z$.
		\begin{enumerate}[label=(\roman*),itemsep=2pt,parsep=2pt,topsep=4pt]
			\item If $\sqrt{d}\notin k(\zeta_n)$, then
			\[  \left(M_d/nM_d(-1) \right)^{\Gamma_k} \cong 
			\begin{cases}
				\z/2, & \text{if $n=2^j$, $j>0$,} \\
				0, & \text{otherwise.} \\
			\end{cases} 
			 \]	
			
			\item If $\sqrt{d}\in k(\zeta_n)$, then
			$$\left(M_d/nM_d(-1) \right)^{\Gamma_k} \cong \{\zeta\in\mu_n(k(\sqrt{d})):\tau(\zeta)=\zeta^{-1}\}$$
			where $\tau\in \Gal(k(\sqrt{d})/k)$ is the non-trivial automorphism.
			
			\item If $\sqrt{d}\in k(\zeta_n)$ and $[k(\zeta_n):k]=[\q(\zeta_n):\q]$, then 
			\[  \left(M_d/nM_d(-1) \right)^{\Gamma_k} \cong  
			\begin{cases}
				\z/2 \text{ or } \z/4, & \text{if  $n=2^j$, $j>1$,} \\
				\z/3, & \text{if  $n=3^j$, $j>0$,} \\
				0, & \text{otherwise.} \\
			\end{cases} 
			\] 
		\end{enumerate}

	\end{lem}
	
	\begin{proof}
		First, note that the action of $\Gamma_k$ on $M_d/nM_d(-1)$ factors through $\text{Gal}(k(\zeta_n,\sqrt{d})/k) $, which we denote by $G_n$, where $\zeta_n$ is some fixed primitive $n$th root of unity. We denote ${\phi\in \z/n(-1)=\Hom(\mu_n, \z/n)}$ that maps $\zeta_n$ to $a$ by $\phi_a$. Let $H_n$ be the $G_n$-module with underlying group $\Hom(\mu_n, \z/n)$ equipped with the following $G_n$-action:
		\[  
		\begin{cases}
			g\cdot \phi \mapsto \{x\mapsto -\phi(g^{-1}x)\}, & \text{ if $g(\sqrt{d})=-\sqrt{d}$,} \\
			g\cdot \phi \mapsto \{x\mapsto \phi(g^{-1}x)\}, & \text{ if $g(\sqrt{d})=\sqrt{d}$,} \\
		\end{cases} 
		 \]
		for $\phi \in \Hom(\mu_n, \z/n)$.
		Then it is easy to see that we have an isomorphism of $G_n$-modules $$M_d/nM_d \otimes \Hom(\mu_n, \z/n) \xrightarrow{\sim} H_n.$$  
		
		(i) Suppose that $\sqrt{d}\notin k(\zeta_n)$. Then ${G_n=\text{Gal}(k(\zeta_n,\sqrt{d})/k)\cong \Gal(k(\zeta_n)/k) \times \z/2}$. If ${g=(0,1)\in  \Gal(k(\zeta_n)/k) \times \z/2}$, then $g\cdot \phi_a =\phi_{-a}$; therefore, if $g$ fixes $\phi_a$, then $\phi_{2a}=0\in H_n$, i.e.\ $\phi_a$ is $2$-torsion. This shows that $H_n^{G_n}=\z/2$ for $n=2^j$, $j\geq 1$, and $H_n^{G_n}=0$ for $n=p^j$, $p>2$, $j\geq 1$.  
		
		(ii) Suppose that $\sqrt{d}\in k(\zeta_n)$. Since $\sqrt{d}\notin k$, we have that ${G_n=\text{Gal}(k(\zeta_n)/k)}$ is a non-trivial subgroup of ${(\z/n)^\times}$. For $g\in G_n$, we have that $g^{-1}\zeta_n=\zeta_n^t$ for some $t\in \Z$ coprime to $n$; we denote $g$ by $g_t$. Now note that $\phi_a\in (H_n)^{\Gamma_k}$ if and only if for all  $g_t\in G_n$ we have
		\begin{equation}\label{eqn:phiacases}
			\begin{cases}
				ta\equiv -a \pmod n, & \text{if } g_t(\sqrt d)=-\sqrt d,\\
				ta\equiv  a \pmod n, & \text{if } g_t(\sqrt d)= \sqrt d.
			\end{cases}
			\end{equation} 
		Thus, if $\phi_a\in (H_n)^{\Gamma_k}$ then $\zeta_a$ is fixed by $\Gal(k(\zeta_n)/k(\sqrt{d}))$ so that $\zeta_a\in k(\sqrt{d})$ and $\tau(\zeta_n^a)=\zeta_n^{-a}$. Finally, note that if $\zeta_a \in \mu_n(k(\sqrt{d}))$ and $\tau(\zeta_n^a)=\zeta^{-a}$ for some $a\in \z$, then \eqref{eqn:phiacases} holds for all $g_t\in G_n$. This completes the proof of (ii). 
		
		(iii) Suppose that $\sqrt{d}\in k(\zeta_n)$ and $[k(\zeta_n):k]=[\q(\zeta_n):\q]$ so that ${G_n=\text{Gal}(k(\zeta_n)/k)\cong (\z/n)^\times}$. This implies that
		$[k(\zeta_m):k]=\varphi(m)$ for every $m\mid n$, where $\varphi$ is the Euler totient function. Thus, if
		$\zeta_m\in k(\sqrt d)$, then $\varphi(m)\leq 2$. Since $n$ is a
		prime power, the only non-trivial possibilities are $m=2,3,4$. The result now follows from (ii).
	\end{proof}

	In the case of number fields, we obtain the following bounds on both the transcendental and the full Brauer group modulo constants.
	\begin{cor}\label{cor:transcendental-bound}
		Let $k$ be a number field, and let $U$ be as in
		Theorem \ref{thm:Main}. Then the following statements hold.
		\begin{enumerate}[label=(\roman*),itemsep=2pt,parsep=2pt,topsep=4pt]
			\item $	|\Br(U)/\Br_1(U)|\leq 6 [k:\mathbb Q]^2.$
			\item $|\Br(U)/\Br_0(U)| \leq 6 |W(E_6)|^6 [k:\mathbb Q]^2.$
		\end{enumerate}
	\end{cor}
	
	\begin{proof}
		(i) Since $\br U/\br_1 U \hookrightarrow (\br \overline{U})^{\Gamma_k}$, by Theorem \ref{thm:Main}, it is
		enough to bound the invariants of
		$\mathbb Q/\mathbb Z(-1)$ and
		$\varinjlim_n\bigl(M_d/nM_d(-1)\bigr)$. Let $\mu_{\infty}(k)$ be the group of roots of unity in $k$. It is easy to see that
		$$
		\left| \bigl(\mathbb Q/\mathbb Z(-1)\bigr)^{\Gamma_k} \right|
		= \left| \mu_{\infty}(k) \right|,
		$$
		Now, by Lemma \ref{lem:FixedByGalAction}(i) and (ii), we have 
		\[
		\left|\varinjlim_n\bigl(M_d/nM_d(-1)\bigr)^{\Gamma_k}\right|
		\leq 
		\left| \{\zeta\in\mu_{\infty}(k(\sqrt{d})):\tau(\zeta)=\zeta^{-1}\} \right|
		\leq \left|\mu_\infty(k(\sqrt{d}))\right| .
		\]
		Since $k$ is a number field, the cyclic group $\mu_\infty(k(\sqrt{d}))$ is finite. Let $\zeta_m$ be an $m$-th root of unity generating $\mu_\infty(k(\sqrt{d}))$ so that $\q(\zeta_m)\subset k(\sqrt{d})$. Thus,

		$$\varphi(m)\leq 2[k:\mathbb Q].$$
		For $m\neq 2,6$, we have 
		$ \sqrt{m} \leq \varphi(m).$
		Therefore, we deduce that 
		$$m\leq \max\{6,  4[k:\mathbb Q]^2\}\leq 6 [k:\mathbb Q]^2.$$
		
		(ii) The bound follows by combining (i) and Lemma \ref{lem:BrUisoH1}(ii).
		\end{proof}

	\begin{thm}\label{cor:BrTrPoss}
		Let $X$ be a smooth cubic surface over a field $k$ of characteristic $0$,  $H$ be a hyperplane section, and  $U=X\setminus H$. Assume that
		$[k(\zeta_n):k]=[\q(\zeta_n):\q]$, for every integer $n\geq 1$.
		\begin{enumerate}[itemsep=2pt,parsep=2pt,topsep=4pt]
			\item Suppose that $H$ is the union of a line $\ell$ and a smooth conic $C$ both defined over $k$. If $\ell \cap C$ contains a $k$-point, then $\br U/ \br_1 U$ is a subgroup of 
			\[  
			\begin{cases}
				0, & \text{if $\ell$ is tangent to C,} \\
				\z/2, & \text{if $|\ell(k)\cap C(k)|=2$.} \\
				
			\end{cases} 
			 \]
			
			If $\ell$ meets $C$ at a degree $2$ point with residue field $L=k(\sqrt{d})$, $d\notin k^2$, then  $\br U/ \br_1 U$ is a subgroup of
			\[ 
			\begin{cases}
				\z/4, & \text{if $k(\sqrt{d})\subset k(\zeta_{4})$,} \\
				\z/3\times \z/2, & \text{if $k(\sqrt{d})\subset k(\zeta_{3})$,} \\
				\z/2, & \text{otherwise.} \\
			\end{cases} 
			 \]
			
			\item Suppose that $H$ is an irreducible singular cubic curve $C$. If $C$ is cuspidal, then 
			$$ \br U/ \br_1 U \cong 0 .$$
			If $C$ is nodal split multiplicative, then
			$$ \br U/ \br_1 U \subseteq \z/2 .$$
			
			If $C$ is nodal non-split multiplicative where the corresponding quadratic extension is $L=k(\sqrt{d})$, $d\notin k^2$, then $\br U/ \br_1 U$ is a subgroup of
			\[ 
			\begin{cases}
				\z/4, & \text{if $k(\sqrt{d})\subset k(\zeta_{4})$,} \\
				\z/3\times \z/2, & \text{if $k(\sqrt{d})\subset k(\zeta_{3})$,} \\
				\z/2, & \text{otherwise.} \\
			\end{cases} 
			 \]
			
			\item Suppose that $H$ is geometrically the union of three lines and let $L$ be the minimal extension over which they are defined. If the $3$ lines intersect at an Eckardt point, then 
			$$ \br U/ \br_1 U \cong 0 .$$
			
			If $L=k$ or $L/k$ is a degree $3$ Galois extension, then 
			$$\br U/ \br_1 U\subset \z/2.$$
			
			If $L\supset k(\sqrt{d})$ for some $d\notin k^2$,  
			then $\br U/ \br_1 U$ is a subgroup of
			\[ 
			\begin{cases}
				\z/4, & \text{if $k(\sqrt{d})\subset k(\zeta_{4}),$} \\
				\z/3\times \z/2, & \text{if $k(\sqrt{d})\subset k(\zeta_{3}),$} \\
				\z/2, & \text{otherwise.} \\
			\end{cases} 
			\]
		\end{enumerate} 
	\end{thm}
	\begin{proof}
		First, note that $\br U/\br_1 U \hookrightarrow (\br \overline{U})^{\Gamma_k}$. Since the Brauer group is a direct sum of its $p$-primary subgroups, it suffices to compute the Galois invariant subgroup of $\br \overline{U}[n]$ for $n$ a power of $p$, for all primes $p$. The result now follows from Proposition \ref{prop:case1}, Proposition \ref{prop:case2}, Proposition \ref{prop:case3}, and Lemma \ref{lem:FixedByGalAction}.
	\end{proof}
	
	\begin{rem}\label{rem:k=QTrGpPoss}
		In Theorem \ref{cor:BrTrPoss}, if $\Br(\overline{U})^{\Gamma_k}$ is non-trivial, then it always contains a copy of $\z/2$; indeed, this follows by Lemma 
		\ref{lem:FixedByGalAction}(i)  applied to $n=2$.  
	\end{rem}

	\subsection{Proof of Theorem \ref{thm:DescentToQ}.}
	To prove the theorem, we construct explicit examples using the following proposition.
	\begin{prop}\label{prop:example}
		Let $H\subset \mathbb{P}^2_{\q}$ be a $\q$-subscheme that is geometrically the union of three lines with three distinct pairwise intersection points. Let $L$ be the splitting field of $H$, and write $H_L=\cup_{i=1}^3 \ell_i $ where each $\ell_i$ is given by a linear form $f_i$. Write $\ell_i^{\circ}:= \ell_i\setminus (\ell_j\cup \ell_k)$ for distinct $i,j,k$. Let $\{P_1, \ldots, P_6\} \subset \mathbb{P}^2_{L}(L)$ be a $\Gal(L/\q)$-stable set of points in general position such that $P_i,P_{i+3} \in \ell_i^{\circ}$ for $i=1,2,3$. Then the following hold.
		\begin{enumerate}[itemsep=2pt,parsep=2pt,topsep=4pt]
			\item Let $S$ be the singular locus of $H$ so that $S=\spec(E)$ for some étale $\q$-algebra $E$ of degree $3$. Then $T:=\mathbb{P}_{\q}^2\setminus H$ is isomorphic to the torus $(\Res_{E/\q} \gm) / \gm$.
			\item Let $\pi:X\rightarrow \mathbb{P}_{\q}^2$ be the blow up of $\mathbb{P}^2_{\q}$ at the $\q$-subscheme $\{P_1,\ldots, P_6\}$ and let $\widetilde{H}$ be the strict transform of $H$. Then $X$ is a smooth cubic surface over $\q$ such that  $\widetilde{H}$ is a hyperplane section that is geometrically the union of three lines not intersecting at an Eckardt point.
			\item Let $e_i$ be the exceptional divisor above the point $P_i$, which is defined over $L$. Let $U:= X\setminus \widetilde{H}$ and $V:=U \setminus \cup_{i=1}^6 e_i$, both defined over $\q$. Then $\pi$ induces an isomorphism $V\xrightarrow{\sim} T$ giving a commutative diagram
			\begin{equation}\label{diag:br}
				\begin{tikzcd}[column sep=4.5em]
					{\Br(V_L)} &  {\bigoplus\limits_{i=1}^3\Ho^1(L(e_i), \mathbb{Q}/\mathbb{Z})  \oplus \Ho^1(L(e_{i+3}), \mathbb{Q}/\mathbb{Z}) \oplus \Ho^1(L(\widetilde{\ell_i}), \mathbb{Q}/\mathbb{Z})} \\
					{\Br(T_L)} & {\bigoplus\limits_{i=1}^3 \Ho^1(\ell_i^{\circ}, \mathbb{Q}/\mathbb{Z})}
					\arrow["{\{\partial_{e_i}, \partial_{e_{i+3}}, \partial_{\widetilde{\ell_i}}\}}", from=1-1, to=1-2]
					\arrow["{\pi^*}", from=2-1, to=1-1]
					\arrow["{\{\partial_{\ell_i}\}}", from=2-1, to=2-2]
					\arrow["{\bigoplus_{i=1}^3 (r_i \oplus \pi^*)}",from=2-2, to=1-2]
				\end{tikzcd}
			\end{equation}
			where the horizontal maps are the (Gysin) residue maps and $$r_i: \Ho^1(\ell_i^{\circ}, \mathbb{Q}/\mathbb{Z})\rightarrow \Ho^1(L(e_i), \mathbb{Q}/\mathbb{Z})  \oplus \Ho^1(L(e_{i+3}), \mathbb{Q}/\mathbb{Z})$$ is given by evaluating the residue at the points $P_i$ and $P_{i+3}$  followed
			by restriction.

			\item If $L=\q$, then $(\Br(T)/\Br_1(T))[2]\cong \z/2$ and is generated by
			$$A_2:=\left( \frac{f_1}{f_3},   \frac{f_2}{f_3}\right).$$
			If moreover $r_i(\partial_{\ell_i}(A_2))=0$ for all $i\in \{1,2,3\}$, then $$\pi^*(A_2)\in \Br(U)\setminus \Br_1(U).$$
			\item Suppose that $L=\q(\omega)$, where $\omega$ is a primitive
			cube root of unity. Let $\sigma$ be the nontrivial element
			of $\Gal(L/\q)$, and choose $f_1,f_2,f_3$ so that
			$f_3$ is defined over $\q$ and $\sigma(f_1)=f_2$. Then $(\Br(T)/\Br_1(T))[3]\cong \z/3$ and is generated by
			$$A_3:=\Cor_{L(T_L)/\q(T)} \left( \frac{f_1}{f_3},   \frac{f_2}{f_3}\right)_{\omega}.$$
			If moreover $r_i(\partial_{\ell_i}(f_1/f_3,f_2/f_3)_\omega)=0$ for all $i\in \{1,2,3\}$, then $$\pi^*(A_3)\in \Br(U)\setminus \Br_1(U).$$
		\end{enumerate}
	\end{prop}
	\begin{proof}
		(1), (2) easily follow.
		
		(3) Since
		$$
		\pi^*H=\sum_{i=1}^3\widetilde{\ell_i}+\sum_{i=1}^6e_i,
		$$
		we have $V=X\setminus\pi^{-1}(H)$, and $\pi$ restricts
		to an isomorphism $V\xrightarrow{\sim}T$. Now the residue maps in the bottom row of \eqref{diag:br} lie in $\Ho^1(\ell_i^\circ,\mathbb Q/\mathbb Z)$ since each $\ell_i^\circ$ is regular. The inverse image of $\ell_i^\circ$ under
		$\pi^{-1}(\mathbb{P}_{L}^2\setminus S)\to \mathbb{P}_{L}^2\setminus S$ has components
		$e_i$, $e_{i+3}$ and the strict transform of ${\ell_i^\circ}$,
		each of multiplicity one. Note that the function fields of the strict transforms of $\ell_i^\circ$ and ${\ell_i}$ are equal. Thus, by the functoriality of residues the diagram \eqref{diag:br} commutes. 
		
		It remains to show that $r_i$ is the claimed map.
		For $j=i,i+3$, the restriction of $\pi$ to $e_j$ factors as
		$$
		e_j\longrightarrow \spec L
		\xrightarrow{P_j}\ell_i^\circ.
		$$
		Consequently, by functoriality
		of residues, the corresponding map is the composition
		$$
		\Ho^1(\ell_i^\circ,\mathbb Q/\mathbb Z)
		\xrightarrow{P_j^*}
		\Ho^1(L,\mathbb Q/\mathbb Z)
		\xrightarrow{\res_{L(e_j)/L}}
		\Ho^1(L(e_j),\mathbb Q/\mathbb Z).
		$$
		Since $P_j^*$ is evaluation at $P_j$, this identifies
		$r_i$ with evaluation at $P_i$ and $P_{i+3}$ followed
		by restriction, as claimed.

		(4) If $L=\q$, then $T\cong \gm^2$. By \cite[Theorem 1.1]{Demeio}, we have
		$$\Br(T)/\Br_1(T)\cong (\q/\z(-1))^{\Gamma_{\q}} \cong \z/2.$$ 
		Note that $A_2\in \Br(T)\setminus \Br_1(T)$, since $A_2$ ramifies only along $\ell_1, \ell_2$ and $\ell_3$ with non-algebraic residues by construction. Now by purity 
		(see e.g.\ \cite[Theorem 3.7.3]{CTSK}), we have an exact sequence
		$$0\rightarrow \Br U \rightarrow \Br V \rightarrow \bigoplus_{i=1}^6 \Ho_{\text{ét}}^1 (\q(e_i), \q/\z).$$
		Therefore, by the commutative diagram \eqref{diag:br} in (3), we deduce that $\pi^*A_2 \in \Br(U)$ if and only if $r_i(\partial_{\ell_i}(A_2))=0$ for $i=1,2,3$, as desired.
		
		(5) Suppose that $L=\q(\omega)$, and let $\sigma$ be the generator of $\Gal(L/\q)$. Then $T_L\cong \mathbb{G}_{m,L}^2$ and by \cite[Theorem 1.1]{Demeio}, we have
		$$\Br(T_L)/\Br_1(T_L)\cong (\q/\z(-1))^{\Gamma_{L}} \cong \z/6.$$ 
		Since $L$ is a degree $2$ extension of $\q$, the restriction map $\res_{L(T_L)/\q(T)}$ is injective on the $3$-torsion part so that $(\Br(T)/\Br_1(T))[3]\subset \z/3$.
		Similar to (4), we have that ${(f_1/f_3,f_2/f_3)_\omega\in \Br(T_L)\setminus \Br_1(T_L)}$. Since $[L:\q]=2$ and  $(f_1/f_3,f_2/f_3)_\omega$ can only ramify along $\ell_1, \ell_2$ or $\ell_3$, we deduce that $A_3\in \Br(\q(\mathbb{P}^2))$ can only ramify along the $\q$-irreducible components of $H$. Note that
		\begin{equation*}
			\begin{split}
				\res_{L(\mathbb{P}^2)/\q(\mathbb{P}^2)} A_3 &= \left( \frac{f_1}{f_3},   \frac{f_2}{f_3}\right)_{\omega}+ \sigma\left( \frac{f_1}{f_3},   \frac{f_2}{f_3}\right)_{\omega}\\ &= \left( \frac{f_1}{f_3},   \frac{f_2}{f_3}\right)_{\omega}+\left( \frac{f_2}{f_3},   \frac{f_1}{f_3}\right)_{\sigma(\omega)}= 2 \left( \frac{f_1}{f_3},   \frac{f_2}{f_3}\right)_{\omega}
			\end{split}
		\end{equation*}
		where the last equality follows by the fact that $(a,b)_{\sigma(\omega)}=(a,b)_{\omega^{-1}}=-(a,b)_{\omega}$ for $a,b\in L(\mathbb{P}^2)$.
		Now, since $(f_1/f_3,f_2/f_3)_\omega\in \Br(T_L)\setminus \Br_1(T_L)$ is a non-trivial $3$-torsion element, we deduce that 
		$A_3$ is non-trivial in $(\Br(T)/\Br_1(T))[3]$, i.e.\ $(\Br(T)/\Br_1(T))[3]\cong \z/3$. As the residues of $A_3$ along the irreducible components of $H$ are of order $3$ and $[L:\q]=2$, to show that $A_3\in \Br(U)$ it will suffice to show that $2\cdot (f_1/f_3,f_2/f_3)_\omega\in \Br(U_L)[3]$. Thus,  by the commutative diagram \eqref{diag:br} in (3), this is equivalent to showing that $r_i(\partial_{\ell_i}(f_1/f_3,f_2/f_3)_\omega)=0$ for $i=1,2,3$.
	\end{proof}

	\begin{rem}\label{rem:BrUtrivial}
		Suppose that the points in Proposition \ref{prop:example} can be labelled so that
		$\{P_1,P_2,P_3\}$ is $\gal(L/\q)$-stable. Then
		$$\Br_1(U)=\Br_0(U).$$
		Indeed, writing $\ell$ for the class of the pullback of a line, we have the relation ${\ell=[\pi^* \ell_i] = [\widetilde{\ell_i}] + e_i+e_{i+3}}$ for $i=1,2,3$. Thus, by \eqref{eq:PicUexact}, we have
		$$
		\Pic(\overline U)
		\cong \z \ell\oplus\bigoplus_{i=1}^3\z e_i,
		$$
		which is a permutation module. This implies that $\Ho^1(\q,\Pic(\overline U))=0$ by Shapiro's lemma, so that $\Bra(U)=0$ by  Lemma \ref{lem:BrUisoH1}.
	\end{rem}

	We now use Proposition \ref{prop:example} to produce explicit examples proving Theorem \ref{thm:DescentToQ}. 
	
	\medskip
	\noindent\textbf{Case I: Constructing $U$ with a transcendental element of order $2$.}
	\par\smallskip
	
	Following the construction in Proposition \ref{prop:example}, by (4), we need to find three homogeneous  linear forms $f_1,f_2, f_3\in \q[X,Y,Z]$ and a choice of six points $\{P_1,\ldots,P_6\}\subset \mathbb{P}^2_{\q}(\q)$ in general position such that $P_i,P_{i+3} \in \ell_i$ and $r_i(\partial_{\ell_i}(f_1/f_3,f_2/f_3))=0$ for $i=1,2,3$. By the residue formula (see e.g.\ \cite[Theorem 1.4.14]{CTSK}), we compute
	\[  \partial_{\ell_i}\left( \frac{f_1}{f_3},   \frac{f_2}{f_3}\right) = 
	\begin{cases}
		\class(f_2/f_3) \in \q(\ell_1)^\times/(\q(\ell_1)^\times)^2, & \text{if $i=1$,} \\
		\class(f_1/f_3) \in \q(\ell_2)^\times/(\q(\ell_2)^\times)^2, & \text{if $i=2$,} \\
		\class(-f_1/f_2) \in \q(\ell_3)^\times/(\q(\ell_3)^\times)^2, & \text{if $i=3$.} 
	\end{cases} 
	\] 
	Therefore, to have that $r_i(\partial_{\ell_i}(f_1/f_3,f_2/f_3))=0$ for $i=1,2,3$, by Proposition \ref{prop:example}(3) we require that
	$$\left\{\frac{f_2}{f_3}(P_1), \frac{f_2}{f_3}(P_4), \frac{f_1}{f_3}(P_2), \frac{f_1}{f_3}(P_5), -\frac{f_1}{f_2}(P_3), -\frac{f_1}{f_2}(P_6) \right\} \subset {\q^\times}^2.$$ 
	Choose $f_1=X$, $f_2=Y$, and $f_3=Z$. Write $P_i=[0:Y_i:Z_i]$ for $i=1,4$; $P_i=[X_i:0:Z_i]$ for $i=2,5$; and $P_i=[X_i:Y_i:0]$ for $i=3,6$. We need any choice of $\{P_i\}$ so that they lie in general position, and such that
	$$\{Y_1/Z_1, Y_4/Z_4, X_2/Z_2, X_5/Z_5, -X_3/Y_3, -X_6/Y_6\} \subset {\q^\times}^2.$$
	One easily verifies that the following choices work:
	${P_1=[0:1:1],}$ ${P_4=[0:4:1],}$  ${P_2=[9:0:1],}$ ${P_5=[16:0:1],}$ ${P_3=[-25:1:0],}$ and $P_6=[-36:1:0].$ Therefore, applying Proposition \ref{prop:example} to the chosen forms $f_1,f_2,f_3$ and the points $P_1,\ldots ,P_6$, we obtain the desired surface $U$.  By Remark \ref{rem:BrUtrivial}, this completes the proof of Theorem \ref{thm:DescentToQ} for the case $n=2$.

	\medskip
	\noindent\textbf{Case II: Constructing $U$ with a transcendental element of order $3$.}
	\par\smallskip
	
	Let $L=\q(\omega)$. Following the construction in Proposition \ref{prop:example}, by (5), we need to find three homogeneous  linear forms $f_1,f_2, f_3\in L[X,Y,Z]$ and a  $\Gal(L/\q)$-stable set of six points $\{P_1,\ldots,P_6\}\subset \mathbb{P}^2_{L}(L)$  in general position such that $P_i,P_{i+3} \in \ell_i$ and $r_i(\partial_{\ell_i}(f_1/f_3,f_2/f_3)_\omega)=0$ for $i=1,2,3$. By the residue formula (see e.g.\ \cite[Theorem 1.4.14]{CTSK}), we compute
	\[  \partial_{\ell_i}\left( \left( \frac{f_1}{f_3},   \frac{f_2}{f_3}\right)_{\omega} \right) =  
	\begin{cases}
		\class(f_2/f_3) \in L(\ell_1)^\times/(L(\ell_1)^\times)^3, & \text{if $i=1$,} \\
		\class(f_3/f_1) \in L(\ell_2)^\times/(L(\ell_2)^\times)^3, & \text{if $i=2$,} \\
		\class(f_1/f_2) \in L(\ell_3)^\times/(L(\ell_3)^\times)^3, & \text{if $i=3$.} 
	\end{cases} 
	\] 
	Therefore, to have that $r_i(\partial_{\ell_i}(f_1/f_3,f_2/f_3)_\omega)=0$ for $i=1,2,3$, by Proposition \ref{prop:example}(3) we require that
	$$\left\{\frac{f_2}{f_3}(P_1), \frac{f_2}{f_3}(P_4), \frac{f_1}{f_3}(P_2), \frac{f_1}{f_3}(P_5), \frac{f_1}{f_2}(P_3), \frac{f_1}{f_2}(P_6) \right\} \subset {L^\times}^3.$$ 
	Choose $$f_1=X+\omega Y+Z, \> f_2=X+\omega^2 Y+Z \text{ and } f_3=Z.$$ We start by choosing two conjugate points $P_1\in \ell_1$ and $P_2\in \ell_2$ such that $(f_2/f_3)(P_1)$ and $(f_1/f_3)(P_2)$ are cubes in $L$. We claim that the two conjugate points
	\begin{equation*}
		P_1=[-\omega^2:1:\omega^2-\omega] \text{ and } P_2=[-\omega:1:\omega-\omega^2]
	\end{equation*}
	work. Indeed, one verifies easily that $P_1\in \ell_1$, $P_2\in \ell_2$,  $(f_2/f_3)(P_1)=1$ and $(f_1/f_3)(P_2)=1$.
	We now choose a different pair of conjugate points $P_4\in \ell_1$ and $P_5\in \ell_2$ such that $(f_2/f_3)(P_4)$ and $(f_1/f_3)(P_5)$ are cubes in $L$. We claim that the two conjugate points
	\begin{equation*}
		P_4=[(\omega-\omega^2)/8-\omega:1:(\omega^2-\omega)/8] \text{ and } P_5=[(\omega^2-\omega)/{8}-\omega^2:1:(\omega-\omega^2)/{8}]
	\end{equation*}
	work. Indeed, a straightforward computation shows that $P_4\in \ell_1$, $P_5\in \ell_2$,  $(f_2/f_3)(P_4)=8$ and $(f_1/f_3)(P_5)=8$.
	
	Finally, we choose two $\q$-points $P_3,P_6\in \ell_3$ such that $\{P_1,\ldots, P_6\}$ are in general position, and so that $(f_1/f_2)(P_3)$ and $(f_1/f_2)(P_6)$ are cubes in $L$. We claim that the choices 
	\begin{equation*}
		P_3=[37/20:1:0] \text{ and } P_6=[1:0:0]
	\end{equation*}
	work. Indeed, by a straightforward computation, one verifies that $$(f_1/f_2)(P_3)= \left( \frac{8+3\omega}{7} \right)^3 \text{ and } (f_1/f_2)(P_6)=1.$$
	Moreover, one can verify that the points $\{P_1,\ldots, P_6\}$ are indeed in general position (e.g.\ by Magma). By Remark \ref{rem:BrUtrivial}, this completes the proof of Theorem \ref{thm:DescentToQ} for the case $n=3$. \hfill \qed

	\subsection{Transcendental Brauer--Manin obstructions} 
		Let $U$ be a smooth affine surface over $\q$ and let $\mathcal{U}$ be a $\z$-model for $U$. We denote $\mathcal{U}(\mathbb{A}_{\z}):=\prod_{p} \mathcal{U}(\z_p),$ where $p$ runs through all primes and $\infty$, and $\z_{\infty} = \mathbb{R}$.
		We let $$\mathcal{U}(\mathbb{A}_{\z})_{\bullet}:= \pi_0(U(\mathbb{R}))\times \prod_{p\neq \infty} \mathcal{U}(\z_p),$$ where $\pi_0(U(\mathbb{R}))$  is the set of connected components of $U(\mathbb{R})$.
		If $\mathcal{U}(\mathbb{A}_{\z})\neq \emptyset$ and $\mathcal{U}(\z)=\emptyset$, we say that the \textit{integral Hasse principle} fails for $\mathcal{U}$.  An adaptation of the Brauer--Manin obstruction by Colliot-Thélène and Xu \cite{CTXU} extended the use of the Brauer group to study the integral Hasse principle. More precisely, one may compute $\mathcal{U}(\mathbb{A}_{\z})^{\Br}$, the subset of $\mathcal{U}(\mathbb{A}_{\z})$ that is orthogonal to $\Br(U)$, and whenever this set is empty and $\mathcal{U}(\mathbb{A}_{\z})\neq \emptyset$ one says that there is an \textit{integral Brauer--Manin obstruction to the integral Hasse principle on $\mathcal{U}$.} Moreover, if $\mathcal{U}(\mathbb{A}_{\z})^{\Br}\subsetneq \mathcal{U}(\mathbb{A}_{\z})_{\bullet}$ then we say that there is an \textit{integral Brauer--Manin obstruction to strong approximation on $\mathcal{U}$}.
		
		Take $U$ to be the affine cubic surface constructed using Proposition \ref{prop:example}. Let $\mathcal{H}, \mathcal{P}$ be the closures of $H, \{P_1, \ldots, P_6\}$ in $\mathbb{P}^2_{\z}$ respectively. Then ${\mathcal{T}:= \mathbb{P}^2_{\z} \setminus \mathcal{H}}$ and ${\mathcal{U}:= (\Bl_{\mathcal{P}} \mathbb{P}^2_\z )\setminus \widetilde{\mathcal{H}}}$ are $\z$-models for $T$ and $U$, respectively, where $\widetilde{\mathcal{H}}$ is the strict transform of ${\mathcal{H}}$. The inverse image of $\mathcal T$ via the blow-up morphism is an open subscheme of $\mathcal U$, which gives an open immersion  $\mathcal T\hookrightarrow\mathcal U$.
	
	\begin{rem}
		The affine cubic surfaces constructed using Proposition \ref{prop:example} are $\mathbb{Q}$-rational and hence satisfy weak approximation. We have seen in Proposition \ref{prop:example} that $T$ is a torus over $\q$. If the model $\mathcal{T}:= \mathbb{P}^2_{\z} \setminus \mathcal{H}$ is also a group scheme over $\z$, then $\mathcal{U}(\z) \supset \mathcal{T}(\z)\neq \emptyset$; thus, the integral Hasse principle holds. For the examples constructed in Cases I and II, one may verify that $\mathcal{T}\cong \mathbb{G}^2_{m,\z}$ and $\mathcal{T}\cong \Res_{\z[\omega]/\z} \mathbb{G}_{m}$, respectively. Thus, no Brauer class can obstruct the integral Hasse principle on $\mathcal{U}$; note that in both examples  $\Br(U)/\Br(\q)$ is non-trivial and fully transcendental as $\Br(X)=\Br_1(U)=\Br(\q)$ (see Remark \ref{rem:BrUtrivial}).
		
		The aforementioned observations are particular to our construction (Proposition \ref{prop:example}), and do not hold in general for affine cubic surfaces where the boundary is geometrically the union of three lines. Indeed, there are examples of Markoff surfaces that fail weak approximation, and ones with Brauer--Manin obstructions to the integral Hasse principle; moreover, these obstructions can be explicit and do arise from algebraic classes in $\Br(U)\setminus\Br(X)$ (e.g., see \cite{CTWX,MarkoffSurf}).
	\end{rem}
	
	Let $\inv_p:\Br(\q_p)\rightarrow \q/\z$ be the local invariant map at $p$. For $p=\infty$, we sometimes write $\q_\infty$ for $\mathbb{R}$. For $A\in \Br(U)$ and $x_p\in \mathcal{U}(\z_p) \subset U(\q_p)$, we define $A(x_p)\in \Br(\q_p)$ to be the pullback of $A$ along the map $x_p\rightarrow U$.
	
	\begin{prop}\label{prop:BrTrSa}
		Let $U$ be the affine cubic surface constructed in Case II
		of the proof of Theorem \ref{thm:DescentToQ}, and let
		$A:=\pi^*A_3\in\Br(U)[3]$ be the transcendental class
		constructed there. Let $p$ be a prime such that $p\equiv 1 \mod 3$. Then the following statements hold.
		\begin{enumerate}[label=(\roman*),itemsep=2pt,parsep=2pt,topsep=4pt]
			\item The local invariant map
			$$
			U(\q_p) \longrightarrow (1/3)\z/\z,
			\qquad R\longmapsto\inv_p A(R),
			$$
			is nonconstant.
			
			\item There exists an integral model $\mathcal{U}_p$
			of $U$ on which $A$ gives an integral Brauer--Manin obstruction
			to strong approximation, i.e., $$\mathcal{U}_p(\mathbb{A}_{\z})^{A} \subsetneq \mathcal{U}_p(\mathbb{A}_{\z})_{\bullet}.$$
			
			\item We have $\Br_1(U)=\Br_0(U)$. Thus, the
			obstruction in (ii) cannot be explained by algebraic
			Brauer classes.
		\end{enumerate}
	\end{prop}
	
	\begin{proof}
		(i) Since $p \equiv 1 \mod 3$, the field $\q_p$ contains a primitive cube root
		of unity and we obtain an embedding
		$\q(\omega)\hookrightarrow\q_p$. Thus, the torus $T$ splits over $\q_p$ and we have an isomorphism
		$$
		T_{\q_p}\xrightarrow{\sim}{\mathbb{G}_{m,\q_p}^2},
		\qquad
		R\longmapsto(f_1(R)/f_3(R),f_2(R)/f_3(R)).
		$$
		Let $R_1,R_p\in T(\q_p)$ be the inverse image of the points $(1,1), (u
		,p) \in \mathbb{G}_{m}^2(\q_p) $, respectively, for some fixed $u\in \z_p^{\times}$ that is not a cube modulo $p$. Note that we can easily see that $R_1=[0:0:1]\in T(\q)$. Since we have $L\otimes_{\q}\q_p\cong \q_p\times\q_p$, and 
		corestriction commutes with base change, we get 
		$$ A(R_p)
		=2(u,p)_\omega,$$ 
		(see \cite[Proposition 3.8.1 and Lemma 3.8.6]{CTSK}). By  \cite[Proposition 1.3.8]{CTSK},
		we have $(u,p)_\omega\neq 0$, since $p$ (having valuation $1$) is not a norm
		from the unramified cubic extension $\q_p(\sqrt[3]{u})$.
		Therefore, since $A(R_1)=\Cor_{L/\q}\left(1,1\right)_\omega=0 \in \Br(\q)$, we have
		\begin{equation*}
			\inv_p A(R_1)= 0 \quad \text { and } \quad
			\inv_p A(R_p)=2\inv_p (u,p)_\omega \neq 0,
		\end{equation*}
		which proves (i).
		
		(ii) By Proposition \ref{prop:example}(2), we may choose an anticanonical embedding of $X$ into $\mathbb{P}^3_{\q}(X,Y,Z,W)$ which maps $X$ to the vanishing of a homogeneous degree $3$ polynomial $F(X,Y,Z,W)$ over $\z$, and maps $\widetilde{H}$ to the hyperplane section given by $W=0$. This induces the model of $U$ given by:
		$$
		\{F(x,y,z,1)=0\}\subset\mathbb A^3_{\z}(x,y,z).
		$$
		We choose a positive integer $N_p$ such that both $R_1$ and $R_p$ defined in (i) are mapped to a $\z$-point $R_1'$ and a $\z_p$-point $R_p'$, respectively, on the $\z$-model
		$$
		\mathcal{U}_p=\{N_p^3 F(x'/N_p,y'/N_p,z'/N_p,1)=0\}\subset\mathbb A^3_{\z}(x',y',z')
		$$
		where $x'=N_p x,\> y'=N_p y,$ and $z'=N_p z$. If we want to ensure that the model is flat, we can do so by dividing by the greatest common divisor of the coefficients. Now, we define an integral adelic point $(R'_v)\in \mathcal{U}_p(\mathbb{A}_\z)$ by taking $R'_p$
		at the place $v=p$, and $R'_1$ at every other place.
		Since $A(R'_1)=0$, by (i) we obtain
		$$\sum_v\inv_v A(R'_v)=\inv_p A(R'_p)\neq 0.$$
		This shows that $(R'_v)\notin \mathcal{U}_p(\mathbb{A}_{\z})^{A},$ as desired.
		
		(iii) follows by Remark \ref{rem:BrUtrivial}.
	\end{proof}

	\section{An integral Brauer--Manin obstruction on a non-diagonal affine cubic surface}\label{sec:example}

	Recall the $\z$-scheme 
	\begin{equation}\label{eqn:defnU}
		\mathcal{U}:\> 9 x^3+ y^3 = z^2 + 3 \quad \subset \mathbb{A}^3_{\z}.
	\end{equation}
	This is a non-diagonal affine cubic surface whose boundary in the obvious compactification  $X\subset \mathbb{P}^3_{\z}$ is geometrically the union of three lines meeting at an Eckardt point. In what follows, we show that there is an integral Brauer--Manin obstruction to the Hasse principle on $\mathcal{U}$ given by an algebraic element of order $3$. 
	
	We start with the following proposition.
	
	\begin{prop}\label{prop:Uexm}
		Let
		$$ X:\> 9 X^3+ Y^3 = Z^2W + 3W^3 \quad \subset \mathbb{P}^3_{\Q}$$
		be the  obvious compactification of $U:=\mathcal{U}_{\q}$. Then the following statements hold.
		\begin{enumerate}[itemsep=2pt,parsep=2pt,topsep=4pt]
			\item $X$ is a smooth cubic surface over $\q$.
			\item The hyperplane section $H:=\{W=0\}$ is a union of three geometric lines with splitting field $\q(\omega,\sqrt[3]{-9})$, where $\omega=(-1+\sqrt{-3})/2$. Moreover, the three lines meet at the Eckardt point $[0:0:1:0]\in X(\q)$.
			\item $U$ is the complement of $H$ in $X$, and $U(\q)\neq \emptyset$.
			\item The map $\Br(X)\rightarrow  \Br(U)$ is an isomorphism so that $\Br(U)/\Br_1(U)=0$. Moreover,  $\Br(U)/\Br(\q) \cong \z/3$ and is generated by the class of the degree $3$ cyclic algebra
			$$B:= \Cor_{L(X_L)/\q(X)}\left( -9, \frac{Z+W\sqrt{-3}}{W} \right)_{\omega} \text{ in } \Br(\q(X))[3],$$
			where $L=\q(\omega)$.
			\item $\mathcal{U}(\mathbb{A}_{\z})\neq \emptyset$.
		\end{enumerate}
	\end{prop}
	\begin{proof}
		(1) and (2) follow easily.
		
		(3) One easily verifies that $[-3:39:1:27]\in U(\q)$.
		
		(4) Since $H$ is the union of three geometric lines meeting at an Eckardt point, by Theorem \ref{thm:Main}(iii) we deduce that $\Br(U)/\Br_1(U)=0$. Since ${\Br(X)/\Br(\q)\cong \Ho^1(\q, \pic(\overline{X}))},$ a straightforward Magma computation shows that $\Br(X)/\Br(\q)\cong \z/3$. By looking at the table in Proposition \ref{prop:AlgBr3lines}(1), we conclude that $\Br(U)/\Br(\q)\cong\z/3$ as well.
		
		We show that $B\in \Br(\q(X))[3]$ generates $\Br(X)/\Br(\q)$. Let $A:=\res_{L(X_L)/\q(X)} B$ in $\Br(L(X_L))$. If ${\partial_D: \Br(\q(X)) \rightarrow \Ho^1(\q(D),\q/\z)}$ is the residue map along an irreducible divisor $D$ of $X$, and ${\partial_V: \Br(L(X_L)) \rightarrow \Ho^1(L(V),\q/\z)}$ is the residue map along an irreducible divisor $V$ of $X_L$ above $D$ of multiplicity $m_V$, then $$m_V \cdot\res_{L(V)/\q(D)}(\partial_D(B)) = \partial_V(A).$$
		As $[L:\q]=2$, we have $m_V\in \{1,2\}$ and $\Cor_{L(V)/\q(D)}\circ \res_{L(V)/\q(D)}$ is multiplication by $1$ or $2$. Since $B$ is a Brauer element of order $3$, we conclude that if $A$ has trivial residue along all irreducible divisors of $X_L$ then $B$ has trivial residue along all irreducible divisors of $X$. Thus, by e.g.\ \cite[Theorem 3.7.3]{CTSK}, we conclude that it is enough to show that $A\in \Br(X_L)$ and that $A\notin \Br(L)$. 
		
		We show that $A\in \Br(X_L)$ by proving that the residue along all irreducible divisors of $X_L$ is trivial. Let $\sigma \in \gal(L/\q)$ be the non-trivial element. We may write 
		\begin{equation}
			\begin{split}
				A &=\left( -9, \frac{Z+W\sqrt{-3}}{W} \right)_{\omega}+\sigma\left( -9, \frac{Z+W\sqrt{-3}}{W} \right)_{\omega} 
				\\&= \left( -9,\frac{Z+W\sqrt{-3}}{W} \right)_{\omega} - \left( -9,\frac{Z-W\sqrt{-3}}{W} \right)_{\omega}
				\\&= \left( -9,\frac{Z+W\sqrt{-3}}{Z-W\sqrt{-3}} \right)_{\omega}
			\end{split}
		\end{equation}
		using that $(a,b)_{\sigma(\omega)}=(a,b)_{\omega^{-1}}=-(a,b)_{\omega}$ for $a,b\in L(X_L)$.
		It is clear by this description of $A$ that it can only ramify along the irreducible divisors ${D_1:=\{Z+W\sqrt{-3}=0\} }$ and $D_2:=\{Z-W\sqrt{-3}=0\}$. Since $-9=(Y/X)^3$ in $L(D_1)$ and $L(D_2)$, by applying the residue formula \cite[(1.18) p. 32]{CTSK} one easily computes the residue to be trivial along $D_1$ and $D_2$. This shows that $A\in \Br(X_L)$. Finally, in the proof of Theorem \ref{thm:IMBO} below we compute that $\mathcal{U}(\mathbb{A}_{\z})^B=\emptyset$ which implies that $B$ is not in the image of $\Br(\q)\hookrightarrow \Br(X)$.

		(5) First, note that $[-3:39:1:27]\in \mathcal{U}({\z}_p)$ for all $p\neq 3$. For $p=3$, if we take $X=0$ and $Y=W=1$ we obtain the equation $Z^2+2=0$. Setting $f(Z):=Z^2+2$ and observing that $f(1)\equiv 0 \mod 3$ and $f'(1)\equiv 2 \mod 3$, we deduce by Hensel's lemma that a solution for $f(Z)=0$ exists over $\z_3$, which implies that $\mathcal{U}(\z_3)\neq \emptyset$.
	\end{proof}
	
	We now prove Theorem \ref{thm:IMBO}. Recall that  
	$$\mathcal{U}(\mathbb{A}_{\z})^B:= \{ x=(x_p) \in \mathcal{U}(\mathbb{A}_{\z})\> :\> \sum_p \inv_p B(x_p) = 0\},$$
	is the set of points in $\mathcal{U}(\mathbb{A}_{\z})$ orthogonal to $B$. 
	For ease of notation, we write $\res_{L/\q}$ for $\res_{L(X_L)/\q(X)}$ in what follows.
	
	\subsection{Proof of Theorem \ref{thm:IMBO}}
	Let $B\in \Br(U)$ be the element in Proposition \ref{prop:Uexm}(4). We want to show that
	$\mathcal{U}(\mathbb{A}_{\z})^B=\emptyset.$
	Let $L=\q(\omega)$ and set $A:=\res_{L/\q} B \in \Br(U_L)\cong \Br(X_L)$. Since $[L:\q]=2$, we have $$\sum_{p} \sum_{v \mid p} \inv_v A(x_p)= 2\cdot \left( \sum_p \inv_p B(x_p) \right) \in (1/3)\z/\z \subset \q/\z$$
	for all $x=(x_p)\in \mathcal{U}(\mathbb{A}_{\z})$.
	Therefore, it suffices to show that $\sum_{p} \sum_{v \mid p} \inv_v A(x_p)\neq 0$ for all $x=(x_p)\in \mathcal{U}(\mathbb{A}_{\z})\subset U_L(\mathbb{A}_L)$. Let $\sigma \in \gal(L/\q)$ be the non-trivial element. 
	Taking $x=X/W$, $y=Y/W$, $z=Z/W$, we may write $$A=\left( -9, z+\sqrt{-3} \right)_{\omega}+\sigma(\left( -9, z+\sqrt{-3} \right)_{\omega})= \left( -9, z+\sqrt{-3} \right)_{\omega} - \left( -9, z-\sqrt{-3} \right)_{\omega}. $$
	
	Let $\mathfrak{o}_L$ be the ring of integers of $L$ and set $\mathcal{U}_L:=\mathcal{U}\times_{\z} \mathfrak{o}_L$. Observe that $\mathcal{U}_L$ has good reduction outside of the primes lying above $2$ or $3$. Moreover, the Brauer element $A$ has good reduction at the primes of $L$ not lying above $2,3$, with respect to the model $\mathcal{U}_L$. Therefore,  for all $v \nmid 2, 3$ and $x_v\in \mathcal{U}_L(\mathfrak{o}_v)$ we have $A(x_v)\in \Br(\mathfrak{o}_v)=0$, which implies that $\inv_v A(x_v)=0$;  in particular, this holds for $x_p\in \mathcal{U}(\z_p)$ for $p\neq 2,3$. Now, since $-9$ is a cube in $\q_2$ we see that $A=0$ in $\Br(X_{L_2})$, where $L_2$ is the completion of $L$ at the place above $2$. Therefore, we have $$\sum_{p} \sum_{\mathfrak{p} \mid p} \inv_{\mathfrak{p}} A(x_p)= \inv_{\mathfrak{p}}\left( -9, z+\sqrt{-3} \right)_{\omega} - \inv_{\mathfrak{p}}\left( -9, z-\sqrt{-3} \right)_{\omega} $$
	where $\mathfrak{p}$ is the prime of $L$ above $3$. We identify $(1/3)\z/\z$ with $\z/3$ via multiplication by $3$ and we write $\left( -9, z+\sqrt{-3} \right)_{\omega,\mathfrak{p}}= 3 \cdot \inv_{\mathfrak{p}}\left( -9, z+\sqrt{-3} \right)_{\omega}$ for the image of $\inv_{\mathfrak{p}}\left( -9, z+\sqrt{-3} \right)_{\omega}$ in $\z/3$.
	
	We now use the formulas in \cite[\S 4, p. 34]{CKS} to compute the symbols $\left( -9, z+\sqrt{-3} \right)_{\omega,\mathfrak{p}}$ and $\left( -9, z-\sqrt{-3} \right)_{\omega,\mathfrak{p}}$; we note that the convention in \cite[\S 4, p. 34]{CKS} differs from ours by a minus sign, which we incorporate into our computations. Set $\lambda = 1-\omega$, which is a uniformizer for $\mathfrak{o}_{\mathfrak{p}}$, the ring of integers of the local field $L_{\mathfrak{p}}$. Then by \cite[(68), p. 34]{CKS} we have
	\begin{equation}\label{form1}
		-\left( \omega, 1+a \lambda + b\lambda^2 +\cdots \right)_{\omega,\mathfrak{p}}= a+ a^2 + b \in \z/3
	\end{equation}
	and by \cite[(70), p. 34]{CKS} we have
	\begin{equation}\label{form2}
		-\left( \lambda, 1+a \lambda + b\lambda^2 +c\lambda^3 + \cdots \right)_{\omega,\mathfrak{p}}= (a-a^3)/3 + ab - c  \in \z/3.
	\end{equation}
	
	Let $x_3=(x,y,z)\in \mathcal{U}(\z_3)$. Then by reducing \eqref{eqn:defnU} mod $9$, and noting that the only cubes mod $9$ are $\{0,1,8\}$, we see that $z^2 \equiv 5,6,7 \mod 9$. Since only $7$ is a quadratic residue modulo $9$, we deduce that $z\equiv \pm 4 \mod 9$. Now note that $\sqrt{-3}=2\omega+1=3-2\lambda$, and by \cite[(72), p. 34]{CKS} we have  $3=-\omega^2 \lambda^2$. Therefore, we deduce that either $z+\sqrt{-3}\equiv -1-2\lambda \mod \lambda^4$ or $$z+\sqrt{-3}\equiv 7-2\lambda=1+2(3)-2\lambda = 1-2\omega^2\lambda^2-2\lambda = 1-2\lambda -2\lambda^2+\lambda^3 \mod \lambda^4.$$
	As $-9= -\omega^4 \lambda^4 \equiv \omega \lambda \in L_{\mathfrak{p}}^*/ (L_{\mathfrak{p}}^*)^3$ we have that $$\left( -9, z+\sqrt{-3} \right)_{\omega,\mathfrak{p}}=\left(\omega \lambda , z+\sqrt{-3} \right)_{\omega,\mathfrak{p}}= \left(\omega  , z+\sqrt{-3} \right)_{\omega,\mathfrak{p}}+ \left( \lambda , z+\sqrt{-3} \right)_{\omega,\mathfrak{p}},$$
	and similarly for $\left( -9, z-\sqrt{-3} \right)_{\omega,\mathfrak{p}}$.
	
	Suppose that $z+\sqrt{-3}\equiv-1-2\lambda \mod \lambda^4$. Since $-1-2\lambda= -(1+2\lambda) \mod \lambda^4$ and $-1$ is a cube in $\mathfrak{o}_{\mathfrak{p}}$, by applying formula \eqref{form1}  we compute 
	\begin{equation}\label{eqn:-1,omega}
		\left(\omega  , z+\sqrt{-3} \right)_{\omega,\mathfrak{p}}=\left(\omega , 1+2\lambda + 0 \lambda^2 + 0 \lambda^3+ \cdots \right)_{\omega,\mathfrak{p}}= 0 \in \z/3
	\end{equation}
	and by applying formula \eqref{form2} we compute 
	\begin{equation}\label{eqn:-1,lambda}
		\left(\lambda  , z+\sqrt{-3} \right)_{\omega,\mathfrak{p}}=\left( \lambda , 1+2\lambda+ 0 \lambda^2 + 0\lambda^3 + \cdots \right)_{\omega,\mathfrak{p}}=-1 \in \z/3.
	\end{equation}
	This shows that if $z\equiv-4 \mod 9$, then $\left( -9, z+\sqrt{-3} \right)_{\omega,\mathfrak{p}}=0-1=-1 \in \z/3$.
	
	Suppose that $z+\sqrt{-3}\equiv 1-2\lambda -2\lambda^2+\lambda^3 \mod \lambda^4$. By applying formula \eqref{form1}  we compute 
	\begin{equation}\label{eq:7,omega}
		\left(\omega  , z+\sqrt{-3} \right)_{\omega,\mathfrak{p}}=\left(\omega , 1-2\lambda -2\lambda^2+\lambda^3 + \cdots \right)_{\omega,\mathfrak{p}}= 0 \in \z/3
	\end{equation}
	and by applying formula \eqref{form2} we compute 
	\begin{equation}\label{eq:7,lambda}
		\left(\lambda  , z+\sqrt{-3} \right)_{\omega,\mathfrak{p}}=\left( \lambda , 1-2\lambda -2\lambda^2+\lambda^3 + \cdots \right)_{\omega,\mathfrak{p}}= -2 \in \z/3.
	\end{equation}
	This shows that if $z\equiv 4 \mod 9$, then $\left( -9, z+\sqrt{-3} \right)_{\omega,\mathfrak{p}}=0-2=-2 \in \z/3$. To summarize, we have
	\[ \left( -9, z+\sqrt{-3} \right)_{\omega,\mathfrak{p}} = 
	\begin{cases}
		2  \in \z/3, & \text{if $z \equiv -4 \mod 9$,} \\
		1  \in \z/3, & \text{if $z \equiv 4 \mod 9$.} \\
	\end{cases} 
	 \]

	Now we compute $\left( -9, z-\sqrt{-3} \right)_{\omega,\mathfrak{p}}$. If $z \equiv -4 \mod 9$, we have $z-\sqrt{-3}\equiv -7+2\lambda = -(7-2\lambda) \mod \lambda^4$. Since $-1$ is a cube in $\mathfrak{o}_{\mathfrak{p}}$ we deduce by \eqref{eq:7,omega} and \eqref{eq:7,lambda} that $\left( -9, z-\sqrt{-3} \right)_{\omega,\mathfrak{p}}= -2 \in \z/3.$
	If $z \equiv 4 \mod 9$, we have ${z-\sqrt{-3}}\equiv 1+2\lambda = {-(-1-2\lambda) \mod \lambda^4}$. Thus, we deduce by \eqref{eqn:-1,omega} and \eqref{eqn:-1,lambda} that ${\left( -9, z-\sqrt{-3} \right)_{\omega,\mathfrak{p}}= -1 \in \z/3.}$ Therefore, we have 
	\[ -\left( -9, z-\sqrt{-3} \right)_{\omega,\mathfrak{p}} =  
	\begin{cases}
		2 \in \z/3, & \text{if $z \equiv -4 \mod 9$,} \\
		1 \in \z/3, & \text{if $z \equiv 4 \mod 9$}. \\
	\end{cases} 
	\]	
	
	Finally, we conclude that 
	\[ \inv_{\mathfrak{p}} A(x_3)= \left( -9, z+\sqrt{-3} \right)_{\omega,\mathfrak{p}} - \left( -9, z-\sqrt{-3} \right)_{\omega,\mathfrak{p}}= 
	\begin{cases}
		1 \in \z/3, & \text{if $z \equiv -4 \mod 9$,} \\
		2 \in \z/3, & \text{if $z \equiv 4 \mod 9$.} \\
	\end{cases} 
	\]	
	This shows that $\sum_{p} \sum_{v \mid p} \inv_v A(x_p)\neq 0$ for all $x=(x_p)\in \mathcal{U}(\mathbb{A}_{\z})$, which completes the proof. \hfill \qed

	\begin{rem}
		In general, a Brauer--Manin obstruction to weak approximation
		on a smooth cubic surface can be converted into an obstruction
		to the integral Hasse principle by choosing an affine open
		and a suitable integral model. However, such a construction does not guarantee either a simple
		defining equation or a boundary that is geometrically the union
		of three lines. The latter requires a Galois-stable triple of
		coplanar lines on the original cubic surface, which need not exist.
		The example above has both features.
	\end{rem}

	\iffalse\vspace{20pt}

	\section*{Declarations}
	
	\noindent\textbf{Data Availability} The author declares that the manuscript has no associated data.
	\vspace{1pt}
	
	\noindent\textbf{Conflict of interest} The author states that there is no conflict of
	interest. The author has no competing interests to declare that are relevant to the content of this article.
	
	\fi

\end{document}